\documentclass[12pt,a4paper,draft]{amsart}
\usepackage{latexsym}
\usepackage{ifthen}
\usepackage[leqno]{amsmath}
\usepackage{enumerate}
\usepackage{calc}
\usepackage{amstext,amsbsy,amsopn,amsthm,amsgen,amsfonts,amscd,amsxtra,upref}
\usepackage{mathrsfs}\usepackage{euscript}\usepackage{amssymb}

\swapnumbers
\theoremstyle{plain}
\newtheorem{Thm}{Theorem}[section]
\newtheorem{Lem}[Thm]{Lemma}
\newtheorem{Cor}[Thm]{Corollary}
\newtheorem{Pro}[Thm]{Proposition}
\newtheorem{Prp}[Thm]{Properties}
\newtheorem{Sub}[Thm]{Sublemma}

\theoremstyle{definition}
\newtheorem{Def}[Thm]{Definition}
\newtheorem{Exm}[Thm]{Example}
\newtheorem{Exs}[Thm]{Examples}

\theoremstyle{remark}
\newtheorem{Rem}[Thm]{Remark}
\newtheorem{Rms}[Thm]{Remarks}
\newtheorem*{Com}{Commentary}

\newcommand{\myEmail}{piotr.niemiec@uj.edu.pl}
\newcommand{\myAddress}{\noindent{}Piotr Niemiec\\{}Jagiellonian University\\{}Institute of Mathematics\\{}%%
   ul. \L{}ojasiewicza 6\\{}30-348 Krak\'{o}w\\{}Poland}
\newcommand{\myData}{\author[P. Niemiec]{Piotr Niemiec}\address{\myAddress}\email{\myEmail}}

\newcommand{\NNN}{\mathbb{N}}
\newcommand{\RRR}{\mathbb{R}}

\newcommand{\CCc}{\CMcal{C}}
\newcommand{\FFf}{\CMcal{F}}

\newcommand{\WWw}{\CMcal{W}}

\newcommand{\AaA}{\EuScript{A}}
\newcommand{\FfF}{\EuScript{F}}

\newcommand{\Bb}{\mathfrak{B}}

\newcommand{\Mm}{\mathfrak{M}}\newcommand{\Nn}{\mathfrak{N}}

\newcommand{\mM}{\mathfrak{m}}

\newcommand{\SECT}[1]{\section{#1}\renewcommand{\theequation}{\thesection-\arabic{equation}}\setcounter{equation}{0}}

\newcounter{help}
\newcommand{\ITE}[3]{\ifthenelse{#1}{#2}{#3}}\newcommand{\ITEE}[3]{\ITE{\equal{#1}{#2}}{#3}{}}

\newcommand{\card}{\operatorname{card}}

\newcommand{\leqsl}{\leqslant}\newcommand{\geqsl}{\geqslant}

\newcommand{\varempty}{\varnothing}\newcommand{\dd}{\colon}

\newcommand{\tfcae}{the following conditions are equivalent:}

\newcommand{\THM}[1]{Theorem~\textup{\ref{thm:#1}}}
\newcommand{\COR}[1]{Corollary~\textup{\ref{cor:#1}}}
\newcommand{\LEM}[1]{Lemma~\textup{\ref{lem:#1}}}\newcommand{\PRO}[1]{Proposition~\textup{\ref{pro:#1}}}

\newenvironment{thm}[1]{\begin{Thm}\label{thm:#1}}{\end{Thm}}\newenvironment{lem}[1]{\begin{Lem}\label{lem:#1}}{\end{Lem}}
\newenvironment{cor}[1]{\begin{Cor}\label{cor:#1}}{\end{Cor}}\newenvironment{pro}[1]{\begin{Pro}\label{pro:#1}}{\end{Pro}}

\newcommand{\bibITEM}[2]{\ITE{\equal{#2}{}}{\bibitem{#1} }{\bibitem[#2]{#1} }}
\newcommand{\BIB}[8]{%% article: hidden label, author, title, journal, number, year, pages, shown label
   \bibITEM{#1}{#8} #2, \textit{#3}, #4{} \textbf{#5} (#6), #7.}
\newcommand{\myBIB}[6][P. Niemiec]{#1, \textit{#2}, #3{}\ITE{\equal{#4}{}}{}{ \textbf{#4}} (#5), #6.}
\newcommand{\BIb}[6]{%% book: hidden label, author, title, publisher + city, year, shown label
   \bibITEM{#1}{#6} #2, \textit{#3}, #4, #5.}
\newcommand{\BiB}[9]{%% article in book: hid.label, author, title, in:, book title, publisher etc, year, pages, shown label
   \bibITEM{#1}{#9} #2, \textit{#3}, #4{} \textit{#5}, #6, #7, #8.}

\newcommand{\myBAPP}[3][P. Niemiec]{%% my paper to appear: [authors], title, comment
   #1, \textit{#2}, #3}

\newcommand{\jRN}[2][]{%%
   \ITEE{#2}{ActaM}{\ITE{\equal{#1}{+}}%%
      {Acta Mathematica}{Acta Math.}}%%
   \ITEE{#2}{ActaMSinES}{\ITE{\equal{#1}{+}}%%
      {Acta Mathematica Sinica (English Series)}{Acta Math. Sin. (Engl. Ser.)}}%%
   \ITEE{#2}{AdvM}{\ITE{\equal{#1}{+}}%%
      {Advances in Mathematics}{Adv. in Math.}}%%
   \ITEE{#2}{ACS}{\ITE{\equal{#1}{+}}%%
      {Applied Categorical Structures}{Appl. Categor. Struct.}}%%
   \ITEE{#2}{ActaSM}{\ITE{\equal{#1}{+}}%%
      {Acta Scientiarum Mathematicarum}{Acta Sci. Math.}}%%
   \ITEE{#2}{AmJM}{\ITE{\equal{#1}{+}}%%
      {American Journal of Mathematics}{Amer. J. Math.}}%%
   \ITEE{#2}{AmMMon}{\ITE{\equal{#1}{+}}%%
      {American Mathematical Monthly}{Amer. Math. Mon.}}%%
   \ITEE{#2}{AnnSciEcNormSupT}{\ITE{\equal{#1}{+}}%%
      {Annales Scientifiques de l'\'{E}cole Normale Sup\'{e}rieure (3)}{Ann. Sci. \'{E}c. Norm. Sup\'{e}r. (3)}}%%
   \ITEE{#2}{AnnM}{\ITE{\equal{#1}{+}}%%
      {Annals of Mathematics}{Ann. Math.}}%%
   \ITEE{#2}{AnnProb}{\ITE{\equal{#1}{+}}%%
      {The Annals of Probability}{Ann. Probab.}}%%
   \ITEE{#2}{AnnPALog}{\ITE{\equal{#1}{+}}%%
      {Annals of Pure and Applied Logic}{Ann. Pure Appl. Logic}}%%
   \ITEE{#2}{APM}{\ITE{\equal{#1}{+}}%%
      {Annales Polonici Mathematici}{Ann. Polon. Math.}}%%
   \ITEE{#2}{ArchM}{\ITE{\equal{#1}{+}}%%
      {Archiv der Mathematik}{Arch. Math.}}%%
   \ITEE{#2}{AttiAccLincRendNat}{\ITE{\equal{#1}{+}}%%
      {Atti della Accademia Nazionale dei Lincei. Rendiconti. Classe di Scienze Fisiche, Matematiche e Naturali}%%
      {Atti Accad. Naz. Lincei Rend. Cl. Sci. Fis. Mat. Nat.}}%%
   \ITEE{#2}{BAMS}{\ITE{\equal{#1}{+}}%%
      {Bulletin of the American Mathematical Society}{Bull. Amer. Math. Soc.}}%%
   \ITEE{#2}{BAustrMS}{\ITE{\equal{#1}{+}}%%
      {Bulletin of the Australian Mathematical Society}{Bull. Austral. Math. Soc.}}%%
   \ITEE{#2}{BLondMS}{\ITE{\equal{#1}{+}}%%
      {Bulletin of the London Mathematical Sociecy}{Bull. Lond. Math. Soc.}}%%
   \ITEE{#2}{BAPolSSSM}{\ITE{\equal{#1}{+}}%%
      {Bulletin de l'Acad\'{e}mie Polonaise des Sciences. S\'{e}rie des Sciences Math\'{e}matiques}%%
      {Bull. Acad. Pol. Sci. S\'{e}r. Sci. Math.}}%%
   \ITEE{#2}{BullSM}{\ITE{\equal{#1}{+}}%%
      {Bulletin des Sciences Math\'{e}matiques}{Bull. Sci. Math.}}%%
   \ITEE{#2}{BullPol}{\ITE{\equal{#1}{+}}%%
      {Bulletin of the Polish Academy of Sciences: Mathematics}{Bull. Pol. Acad. Sci. Math.}}%%
   \ITEE{#2}{CanadJM}{\ITE{\equal{#1}{+}}%%
      {Canadian Journal Mathematics}{Canad. J. Math.}}%%
   \ITEE{#2}{CollectM}{\ITE{\equal{#1}{+}}%%
      {Collectanea Mathematica}{Collect. Math.}}%%
   \ITEE{#2}{CMUC}{\ITE{\equal{#1}{+}}%%
      {Commentationes Mathematicae Universitatis Carolinae}{Comment. Math. Univ. Carolin.}}%%
   \ITEE{#2}{CRParis}{\ITE{\equal{#1}{+}}%%
      {C. R. Paris}{C. R. Paris}}%%
   \ITEE{#2}{CRASParis}{\ITE{\equal{#1}{+}}%%
      {Comptes Rendus de l'Acad\'{e}mie des Sciences. Paris}{C. R. Acad. Sci. Paris}}%%
   \ITEE{#2}{CEurJM}{\ITE{\equal{#1}{+}}%%
      {Central European Journal of Mathematics}{Cent. Eur. J. Math.}}%%
   \ITEE{#2}{CMHelv}{\ITE{\equal{#1}{+}}%%
      {Commentarii Mathematici Helvetici}{Comment. Math. Helv.}}%%
   \ITEE{#2}{CollM}{\ITE{\equal{#1}{+}}%%
      {Colloquium Mathematicum}{Coll. Math.}}%%
   \ITEE{#2}{ComposM}{\ITE{\equal{#1}{+}}%%
      {Compositio Mathematica}{Compos. Math.}}%%
   \ITEE{#2}{CzMJ}{\ITE{\equal{#1}{+}}%%
      {Czechoslovak Mathematical Journal}{Czech. Math. J.}}%%
   \ITEE{#2}{DissM}{\ITE{\equal{#1}{+}}%%
      {Dissertationes Mathematicae}{Dissert. Math.}}%%
   \ITEE{#2}{DANSSSR}{\ITE{\equal{#1}{+}}%%
      {Doklady Akademii Nauk SSSR}{Dokl. Akad. Nauk SSSR}}%%
   \ITEE{#2}{DukeMJ}{\ITE{\equal{#1}{+}}%%
      {Duke Mathematical Journal}{Duke Math. J.}}%%
   \ITEE{#2}{ELA}{\ITE{\equal{#1}{+}}%%
      {The Electronic Journal of Linear Algebra}{Electron. J. Linear Algebra}}%%
   \ITEE{#2}{ExtrM}{\ITE{\equal{#1}{+}}%%
      {Extracta Mathematicae}{Extracta Math.}}%%
   \ITEE{#2}{FM}{\ITE{\equal{#1}{+}}%%
      {Fundamenta Mathematicae}{Fund. Math.}}%%
   \ITEE{#2}{FAA}{\ITE{\equal{#1}{+}}%%
      {Functional Analysis and its Applications}{Funct. Anal. Appl.}}%%
   \ITEE{#2}{FunkAnalPril}{\ITE{\equal{#1}{+}}%%
      {Funktsional'ny\u{\i} Analiz i Ego Prilozheniya}{Funkts. Anal. Prilozh.}}%%
   \ITEE{#2}{GTopA}{\ITE{\equal{#1}{+}}%%
      {General Topology and its Applications}{General Topol. Appl.}}%%
   \ITEE{#2}{HJM}{\ITE{\equal{#1}{+}}%%
      {Houston Journal of Mathematics}{Houston J. Math.}}%%
   \ITEE{#2}{IllinoisJM}{\ITE{\equal{#1}{+}}%%
      {Illinois Journal of Mathematics}{Illinois J. Math.}}%%
   \ITEE{#2}{IndagMP}{\ITE{\equal{#1}{+}}%%
      {Indagationes Mathematicae (Proceedings)}{Indagationes Math. Proc.}}%%
   \ITEE{#2}{IndianaUMJ}{\ITE{\equal{#1}{+}}%%
      {Indiana University Mathematical Journal}{Indiana Univ. Math. J.}}%%
   \ITEE{#2}{InHauEtSPM}{\ITE{\equal{#1}{+}}%%
      {Inst. Hautes \'{E}tudes Sci. Publ. Math.}{Inst. Hautes \'{E}tudes Sci. Publ. Math.}}%%
   \ITEE{#2}{IEOT}{\ITE{\equal{#1}{+}}%%
      {Integral Equations and Operator Theory}{Integr. Equ. Oper. Theory}}%%
   \ITEE{#2}{IsraelJM}{\ITE{\equal{#1}{+}}%%
      {Israel Journal of Mathematics}{Israel J. Math.}}%%
   \ITEE{#2}{JAusMSA}{\ITE{\equal{#1}{+}}%%
      {Journal of the Australian Mathematical Society. Series A}{J. Aust. Math. Soc. Ser. A}}%%
   \ITEE{#2}{JCA}{\ITE{\equal{#1}{+}}%%
      {Journal of Convex Analysis}{J. Convex Anal.}}%%
   \ITEE{#2}{JChinUST}{\ITE{\equal{#1}{+}}%%
      {J. China Univ. Sci. Tech.}{J. China Univ. Sci. Tech.}}%%
   \ITEE{#2}{JFA}{\ITE{\equal{#1}{+}}%%
      {Journal of Functional Analysis}{J. Funct. Anal.}}%%
   \ITEE{#2}{JKoreanMS}{\ITE{\equal{#1}{+}}%%
      {Journal of the Korean Mathematical Society}{J. Korean Math. Soc.}}%%
   \ITEE{#2}{JMAnApp}{\ITE{\equal{#1}{+}}%%
      {J. Math. Anal. Appl.}{J. Math. Anal. Appl.}}%%
   \ITEE{#2}{JOT}{\ITE{\equal{#1}{+}}%%
      {Journal of Operator Theory}{J. Operator Theory}}%%
   \ITEE{#2}{KodaiMSemRep}{\ITE{\equal{#1}{+}}%%
      {Kodai Math. Sem. Rep.}{Kodai Math. Sem. Rep.}}%%
   \ITEE{#2}{LAA}{\ITE{\equal{#1}{+}}%%
      {Linear Algebra and its Applications}{Linear Algebra Appl.}}%%
   \ITEE{#2}{LMLA}{\ITE{\equal{#1}{+}}%%
      {Linear and Multilinear Algebra}{Linear Multilinear Algebra}}%%
   \ITEE{#2}{LNM}{\ITE{\equal{#1}{+}}%%
      {Lecture Notes in Mathematics}{Lecture Notes Math.}}%%
   \ITEE{#2}{MathJap}{\ITE{\equal{#1}{+}}%%
      {Math. Japon.}{Math. Japon.}}%%
   \ITEE{#2}{MLQ}{\ITE{\equal{#1}{+}}%%
      {Mathematical Logic Quarterly}{Math. Log. Q.}}%%
   \ITEE{#2}{MProcCambPhS}{\ITE{\equal{#1}{+}}%%
      {Mathematical Proceedings of the Cambridge Philosophical Society}{Math. Proc. Cambridge Phil. Soc.}}%%
   \ITEE{#2}{MMag}{\ITE{\equal{#1}{+}}%%
      {Mathematics Magazine}{Math. Mag.}}%%
   \ITEE{#2}{MSb}{\ITE{\equal{#1}{+}}%%
      {Matematicheski\u{\i} Sbornik}{Mat. Sb.}}%%
   \ITEE{#2}{MStud}{\ITE{\equal{#1}{+}}%%
      {Matematychni Studi\"{\i}}{Mat. Stud.}}%%
   \ITEE{#2}{MScand}{\ITE{\equal{#1}{+}}%%
      {Mathematica Scandinavica}{Math. Scand.}}%%
   \ITEE{#2}{MAnn}{\ITE{\equal{#1}{+}}%%
      {Mathematische Annalen}{Math. Ann.}}%%
   \ITEE{#2}{MAMS}{\ITE{\equal{#1}{+}}%%
      {Memoirs of the American Mathematical Society}{Mem. Amer. Math. Soc.}}%%
   \ITEE{#2}{MichMJ}{\ITE{\equal{#1}{+}}%%
      {Michigan Mathematical Journal}{Mich. Math. J.}}%%
   \ITEE{#2}{MonatM}{\ITE{\equal{#1}{+}}%%
      {Monatshefte f\"{u}r Mathematik}{Mh. Math.}}%%
   \ITEE{#2}{NonlinA}{\ITE{\equal{#1}{+}}%%
      {Nonlinear Analysis: Theory, Methods \& Applications}{Nonlinear Anal.}}%%
   \ITEE{#2}{NAMS}{\ITE{\equal{#1}{+}}%%
      {Notices of the American Mathematical Society}{Notices Amer. Math. Soc.}}%%
   \ITEE{#2}{OpusM}{\ITE{\equal{#1}{+}}%%
      {Opuscula Mathematica}{Opuscula Math.}}%%
   \ITEE{#2}{PacJM}{\ITE{\equal{#1}{+}}%%
      {Pacific Journal of Mathematics}{Pacific J. Math.}}%%
   \ITEE{#2}{PeriodMHung}{\ITE{\equal{#1}{+}}%%
      {Periodica Mathematica Hungarica}{Period. Math. Hungarica}}%%
   \ITEE{#2}{PAMS}{\ITE{\equal{#1}{+}}%%
      {Proceedings of the American Mathematical Society}{Proc. Amer. Math. Soc.}}%%
   \ITEE{#2}{ProcCambPhS}{\ITE{\equal{#1}{+}}%%
      {Proceedings of the Cambridge Philosophical Society}{Proc. Cambridge Phil. Soc.}}%%
   \ITEE{#2}{ProcImpAcadTokyo}{\ITE{\equal{#1}{+}}%%
      {Proc. Imp. Acad. Tokyo}{Proc. Imp. Acad. Tokyo}}%%
   \ITEE{#2}{ProcKonink}{\ITE{\equal{#1}{+}}%%
      {Proceedings of the Koninklijke Nederlandse Akademie van Wetenschappen}{Nederl. Akad. Wetensch. Proc. Ser. A}}%%
   \ITEE{#2}{PLondMS}{\ITE{\equal{#1}{+}}%%
      {Proceedings of the London Mathematical Society}{Proc. London Math. Soc.}}%%
   \ITEE{#2}{PNatlUSA}{\ITE{\equal{#1}{+}}%%
      {Proceedings of the National Academy of Sciences of the United States of America}{Proc. Natl. Acad. Sci. USA}}%%
   \ITEE{#2}{PublRIMSKyoto}{\ITE{\equal{#1}{+}}%%
      {Publ. Res. Inst. Math. Sci. Kyoto Univ.}{Publ. Res. Inst. Math. Sci.}}%%
   \ITEE{#2}{PWN}{\ITE{\equal{#1}{+}}%%
      {PWN -- Polish Scientific Publishers, Warszawa}{PWN -- Polish Scientific Publishers, Warszawa}}%%
   \ITEE{#2}{RCMP}{\ITE{\equal{#1}{+}}%%
      {Rendiconti del Circolo Matematico di Palermo}{Rend. Circ. Mat. Palermo}}%%
   \ITEE{#2}{RussMS}{\ITE{\equal{#1}{+}}%%
      {Russian Mathematical Surveys}{Russian Math. Surveys}}%%
   \ITEE{#2}{SbM}{\ITE{\equal{#1}{+}}%%
      {Sbornik: Mathematics}{Sb. Math.}}%%
   \ITEE{#2}{SciRepTokyoA}{\ITE{\equal{#1}{+}}%%
      {Science Reports of Tokyo Kyoiku Daigaku, Section A}{Sci. Rep. Tokyo Kyoiku Daigaku Sect. A}}%%
   \ITEE{#2}{SeminProbStras}{\ITE{\equal{#1}{+}}%%
      {S\'{e}minaire de probabilit\'{e}s de Strasbourg}{S\'{e}min. Prob. Strasbourg}}%%
   \ITEE{#2}{SIAMJMAA}{\ITE{\equal{#1}{+}}%%
      {SIAM Journal on Matrix Analysis and Applications}{SIAM J. Matrix Anal. Appl.}}%%
   \ITEE{#2}{SibirMZ}{\ITE{\equal{#1}{+}}%%
      {Sibirski\v{\i} Mat. \v{Z}hurnal}{Sibirsk. Mat. \v{Z}.}}%%
   \ITEE{#2}{SM}{\ITE{\equal{#1}{+}}%%
      {Studia Mathematica}{Studia Math.}}%%
   \ITEE{#2}{TAMS}{\ITE{\equal{#1}{+}}%%
      {Transactions of the American Mathematical Society}{Trans. Amer. Math. Soc.}}%%
   \ITEE{#2}{TohokuMJ}{\ITE{\equal{#1}{+}}%%
      {T\^{o}hoku Mathematical Journal}{T\^{o}hoku Math. J.}}%%
   \ITEE{#2}{TomskUnivRev}{\ITE{\equal{#1}{+}}%%
      {Tomsk Universitet Review}{Tomsk. Univ. Rev.}}%%
   \ITEE{#2}{TopA}{\ITE{\equal{#1}{+}}%%
      {Topology and its Applications}{Topology Appl.}}%%
   \ITEE{#2}{TopMethNA}{\ITE{\equal{#1}{+}}%%
      {Topological Methods in Nonlinear Analysis}{Topol. Methods Nonlinear Anal.}}%%
   \ITEE{#2}{TsukubaJM}{\ITE{\equal{#1}{+}}%%
      {Tsukuba Journal of Mathematics}{Tsukuba J. Math.}}%%
   \ITEE{#2}{UspekhiMN}{\ITE{\equal{#1}{+}}%%
      {Uspekhi Matem. Nauk}{Uspekhi Mat. Nauk}}%%
   }

\newcommand{\paplist}[3][]{%%
   \ITEE{#3}{NIAkhiezer,IMGlazman1993}{%%
      \BIb{#2}{N.I. Akhiezer and I.M. Glazman}%%
         {Theory of Linear Operators in Hilbert Space}%%
         {Dover Publications, Inc., New York}{1993}{#1}}%%
   \ITEE{#3}{RDAnderson1967}{%%
      \BIB{#2}{R.D. Anderson}%%
         {On topological infinite deficiency}%%
         {\jRN{MichMJ}}{14}{1967}{365--383}{#1}}%%
   \ITEE{#3}{RDAnderson,JMcCharen1970}{%%
      \BIB{#2}{R.D. Anderson and J. McCharen}%%
         {On extending homeomorphisms to Fr\'{e}chet manifolds}%%
         {\jRN{PAMS}}{25}{1970}{283--289}{#1}}%%
   \ITEE{#3}{RDAnderson,DWCurtis,JVanMill1982}{%%
      \BIB{#2}{R.D. Anderson, D.W. Curtis, J. van Mill}%%
         {A fake topological Hilbert space}%%
         {\jRN{TAMS}}{272}{1982}{311--321}{#1}}%%
   \ITEE{#3}{RArens,JEells1956}{%%
      \BIB{#2}{R. Arens and J. Eells}%%
         {On embedding uniform and topological spaces}%%
         {\jRN{PacJM}}{6}{1956}{397--403}{#1}}%%
   \ITEE{#3}{NAronszajn,PPanitchpakdi1956}{%%
      \BIB{#2}{N. Aronszajn and P. Panitchpakdi}%%
         {Extension of uniformly continuous transformations and hyperconvex metric spaces}%%
         {\jRN{PacJM}}{6}{1956}{405--439}{#1}}%%
   \ITEE{#3}{KJBabenko1948}{%%
      \BIB{#2}{K.J. Babenko}%%
         {On conjugate functions}%%
         {\jRN{DANSSSR}}{62}{1948}{157--160}{#1}}%%
   \ITEE{#3}{TBanakh1995}{%%
      \BIB{#2}{T.O. Banakh}%%
         {Topology of spaces of probability measures, I}%%
         {\jRN{MStud}}{5}{1995}{65--87 (Russian)}{#1}}%%
   \ITEE{#3}{TBanakh1995a}{%%
      \BIB{#2}{T.O. Banakh}%%
         {Topology of spaces of probability measures, II}%%
         {\jRN{MStud}}{5}{1995}{88--106 (Russian)}{#1}}%%
   \ITEE{#3}{TBanakh1998}{%%
      \BIB{#2}{T. Banakh}%%
         {Characterization of spaces admitting a homotopy dense embedding into a Hilbert manifold}%%
         {\jRN{TopA}}{86}{1998}{123--131}{#1}}%%
   \ITEE{#3}{TBanakh,TNRadul1997}{%%
      \BIB{#2}{T.O. Banakh and T.N. Radul}%%
         {Topology of spaces of probability measures}%%
         {\jRN{SbM}}{188}{1997}{973--995}{#1}}%%
   \ITEE{#3}{TBanakh,TRadul,MZarichnyi1996}{%%
      \BIb{#2}{T. Banakh, T. Radul, M. Zarichnyi}%%
         {Absorbing sets in infinite-dimensional manifolds}%%
         {VNTL Publishers, Lviv}{1996}{#1}}%%
   \ITEE{#3}{TBanakh,IZarichnyy2008}{%%
      \BIB{#2}{T. Banakh and I. Zarichnyy}%%
         {Topological groups and convex sets homeomorphic to non-separable Hilbert spaces}%%
         {\jRN{CEurJM}}{6}{2008}{77--86}{#1}}%%
   \ITEE{#3}{HBecker,ASKechris1996}{%%
      \BIb{#2}{H. Becker and A.S. Kechris}%%
         {The Descriptive Set Theory of Polish Group Actions \textup{(London Math. Soc. Lecture Note Series, vol. 232)}}%%
         {University Press, Cambridge}{1996}{#1}}%%
   \ITEE{#3}{GBeer1993}{%%
      \BIb{#2}{G. Beer}%%
         {Topologies on Closed and Closed Convex Sets \textup{(Mathematics and Its Applications)}}%%
         {Kluwer Academic Publishers, Dordrecht}{1993}{#1}}%%
   \ITEE{#3}{NEBenamara,NNikolski1999}{%%
      \BIB{#2}{N.E. Benamara and N. Nikolski}%%
         {Resolvent tests for similarity to a normal operator}%%
         {\jRN{PLondMS}}{78}{1999}{585--626}{#1}}%%
   \ITEE{#3}{YBenyamini,JLindenstrauss2000}{%%
      \BIb{#2}{Y. Benyamini and J. Lindenstrauss}%%
         {Geometric nonlinear functional analysis I}%%
         {AMS Colloquium Publications 48}{2000}{#1}}%%
   \ITEE{#3}{SKBerberian1974}{%%
      \BIb{#2}{S.K. Berberian}%%
         {Lectures in Functional Analysis and Operator Theory}%%
         {Graduate Texts in Mathematics 15, Springer-Verlag, New York}{1974}{#1}}%%
   \ITEE{#3}{SNBernstein1954}{%%
      \BIb{#2}{S.N. Bernstein}%%
         {Collected Works II}%%
         {Akad. Nauk SSSR, Moscow}{1954 (Russian)}{#1}}%%
   \ITEE{#3}{CzBessaga,APelczynski1972}{%%
      \BIB{#2}{Cz. Bessaga and A. Pe\l{}czy\'{n}ski}%%
         {On spaces of measurable functions}%%
         {\jRN{SM}}{44}{1972}{597--615}{#1}}%%
   \ITEE{#3}{CzBessaga,APelczynski1975}{%%
      \BIb{#2}{Cz. Bessaga and A. Pe\l{}czy\'{n}ski}%%
         {Selected topics in infinite-dimensional topology}%%
         {\jRN{PWN}}{1975}{#1}}%%
   \ITEE{#3}{MBestvina,JMogilski1986}{%%
      \BIB{#2}{M. Bestvina and J. Mogilski}%%
         {Characterizing certain incomplete infinite-dimensional absolute retracts}%%
         {\jRN{MichMJ}}{33}{1986}{291--313}{#1}}%%
   \ITEE{#3}{MBestvina,PBowers,JMogilsky,JWalsh1986}{%%
      \BIB{#2}{M. Bestvina, P. Bowers, J. Mogilsky, J. Walsh}%%
         {Characterization of Hilbert space manifolds revisited}%%
         {\jRN{TopA}}{24}{1986}{53--69}{#1}}%%
   \ITEE{#3}{RBhatia1997}{%%
      \BIb{#2}{R. Bhatia}%%
         {Matrix Analysis}%%
         {Springer, New York}{1997}{#1}}%%
   \ITEE{#3}{GBirkhoff1936}{%%
      \BIB{#2}{G. Birkhoff}%%
         {A note on topological groups}%%
         {\jRN{ComposM}}{3}{1936}{427--430}{#1}}%%
   \ITEE{#3}{MSBirman,MZSolomjak1987}{%%
      \BIb{#2}{M.S. Birman and M.Z. Solomjak}%%
         {Spectral Theory of Self-Adjoint Operators in Hilbert Space}%%
         {D. Reidel Publishing Co., Dordrecht}{1987}{#1}}%%
   \ITEE{#3}{EBishop1961}{%%
      \BIB{#2}{E. Bishop}%%
         {A generalization of the Stone-Weierstrass theorem}%%
         {\jRN{PacJM}}{11}{1961}{777--783}{#1}}%%
   \ITEE{#3}{BBlackadar2006}{%%
      \BIb{#2}{B. Blackadar}{Operator Algebras. Theory of $\CCc^*$-algebras and von Neumann algebras %%
         \textup{(Encyclopaedia of Mathematical Sciences, vol. 122: Operator Algebras and Non-Commutative Geometry III)}}%%
         {Springer-Verlag, Berlin-Heidelberg}{2006}{#1}}%%
   \ITEE{#3}{JBlass,WHolsztynski1972}{%%
      \BIB{#2}{J. Blass and W. Holszty\'{n}ski}%%
         {Cubical polyhedra and homotopy III}%%
         {\jRN{AttiAccLincRendNat}}{53}{1972}{275--279}{#1}}%%
   \ITEE{#3}{FFBonsall,NJDuncan1973}{%%
      \BIb{#2}{F.F. Bonsall and N.J. Duncan}%%
         {Complete Normed Algebras}%%
         {Springer Verlag, Berlin}{1973}{#1}}%%
   \ITEE{#3}{NBourbaki2002}{%%
      \BIb{#2}{N. Bourbaki}%%
         {Lie Groups and Lie Algebras, Chapters 4--6}%%
         {Springer, New York}{2002}{#1}}%%
   \ITEE{#3}{PLBowers1989}{%%
      \BIB{#2}{P.L. Bowers}%%
         {Limitation topologies on function spaces}%%
         {\jRN{TAMS}}{314}{1989}{421--431}{#1}}%%
   \ITEE{#3}{ABrown1953}{%%
      \BIB{#2}{A. Brown}%%
         {On a class of operators}%%
         {\jRN{PAMS}}{4}{1953}{723--728}{#1}}%%
   \ITEE{#3}{ABrown,CKFong,DWHadwin1978}{%%
      \BIB{#2}{A. Brown, C.-K. Fong, D.W. Hadwin}%%
         {Parts of operators on Hilbert space}%%
         {\jRN{IllinoisJM}}{22}{1978}{306--314}{#1}}%%
   \ITEE{#3}{AMBruckner,JBBruckner,BSThomson1997}{%%
      \BIb{#2}{A.M. Bruckner, J.B. Bruckner, B.S. Thomson}%%
         {Real Analysis}%%
         {Prentice-Hall, New Jersey}{1997}{#1}}%%
   \ITEE{#3}{PJCameron,AMVershik2006}{%%
      \BIB{#2}{P.J. Cameron and A.M. Vershik}%%
         {Some isometry groups of Urysohn space}%%
         {\jRN{AnnPALog}}{143}{2006}{70--78}{#1}}%%
   \ITEE{#3}{CCastaing1966}{%%
      \BIB{#2}{C. Castaing}%%
         {Quelques probl\`{e}mes de mesurabilit\'{e} li\'{e}es \`{a} la th\'{e}orie de la commande}%%
         {\jRN{CRParis}}{262}{1966}{409--411}{#1}}%%
   \ITEE{#3}{JAVanCasteren1980}{%%
      \BIB{#2}{J.A. van Casteren}%%
         {A problem of Sz.-Nagy}%%
         {\jRN{ActaSM}}{42}{1980}{189--194}{#1}}%%
   \ITEE{#3}{JAVanCasteren1983}{%%
      \BIB{#2}{J.A. van Casteren}%%
         {Operators similar to unitary or selfadjoint ones}%%
         {\jRN{PacJM}}{104}{1983}{241--255}{#1}}%%
   \ITEE{#3}{XCatepillan,MPtak,WSzymanski1994}{%%
      \BIB{#2}{X. Catepill\'{a}n, M. Ptak, W. Szyma\'{n}ski}%%
         {Multiple canonical decompositions of families of operators and a model of quasinormal families}%%
         {\jRN{PAMS}}{121}{1994}{1165--1172}{#1}}%%
   \ITEE{#3}{RCauty1994}{%%
      \BIB{#2}{R. Cauty}%%
         {Un espace m\'{e}trique lin\'{e}aire qui n'est pas un r\'{e}tracte absolu}%%
         {\jRN{FM}}{146}{1994}{85--99, (French)}{#1}}%%
   \ITEE{#3}{TAChapman1971}{%%
      \BIB{#2}{T.A. Chapman}%%
         {Deficiency in infinite-dimensional manifolds}%%
         {\jRN{GTopA}}{1}{1971}{263--272}{#1}}%%
   \ITEE{#3}{TAChapman1976}{%%
      \BIb{#2}{T.A. Chapman}%%
         {Lectures on Hilbert cube manifolds}%%
         {C.B.M.S. Regional Conference Series in Math. No 28, Amer. Math. Soc.}{1976}{#1}}%%
   \ITEE{#3}{RBChuaqui1977}{%%
      \BIB{#2}{R.B. Chuaqui}%%
         {Measures invariant under a group of transformations}%%
         {\jRN{PacJM}}{68}{1977}{313--329}{#1}}%%
   \ITEE{#3}{JBConway1985}{%%
      \BIb{#2}{J.B. Conway}%%
         {A Course in Functional Analysis}%%
         {Springer-Verlag, New York}{1985}{#1}}%%
   \ITEE{#3}{JBConway2000}{%%
      \BIb{#2}{J.B. Conway}%%
         {A Course in Operator Theory}%%
         {(Graduate Studies in Mathematics, vol. 21) Amer. Math. Soc., Providence}{2000}{#1}}%%
   \ITEE{#3}{GCorach,AMaestripieri,MMbekhta2009}{%%
      \BIB{#2}{G. Corach, A. Maestripieri, M. Mbekhta}%%
         {Metric and homogeneous structure of closed range operators}%%
         {\jRN{JOT}}{61}{2009}{171--190}{#1}}%%
   \ITEE{#3}{MJCowen,RGDouglas1978}{%%
      \BIB{#2}{M.J. Cowen and R.G. Douglas}%%
         {Complex geometry and operator theory}%%
         {\jRN{ActaM}}{141}{1978}{187--261}{#1}}%%
   \ITEE{#3}{DWCurtis1985}{%%
      \BIB{#2}{D.W. Curtis}%%
         {Boundary sets in the Hilbert cube}%%
         {\jRN{TopA}}{20}{1985}{201--221}{#1}}%%
   \ITEE{#3}{MMDay1958}{%%
      \BIb{#2}{M.M. Day}%%
         {Normed Linear Spaces}%%
         {Springer Verlag, Berlin}{1958}{#1}}%%
   \ITEE{#3}{CDellacherie1967}{%%
      \BIB{#2}{C. Dellacherie}%%
         {Un compl\'{e}ment au th\'{e}or\`{e}me de Weierstrass-Stone}%%
         {\jRN{SeminProbStras}}{1}{1967}{52--53}{#1}}%%
   \ITEE{#3}{JJDijkstra1987}{%%
      \BIB{#2}{J.J. Dijkstra}%%
         {Strong negligibility of $\sigma$-compacta does not characterize Hilbert space}%%
         {\jRN{PacJM}}{127}{1987}{19--30}{#1}}%%
   \ITEE{#3}{JJDijkstra1990}{%%
      \BIB{#2}{J.J. Dijkstra}%%
         {Characterizing Hilbert space topology in terms of strong negligibility}%%
         {\jRN{ComposM}}{75}{1990}{299--306}{#1}}%%
   \ITEE{#3}{TDobrowolski,WMarciszewski2002}{%%
      \BIB{#2}{T. Dobrowolski and W. Marciszewski}%%
         {Failure of the Factor Theorem for Borel pre-Hilbert spaces}%%
         {\jRN{FM}}{175}{2002}{53--68}{#1}}%%
   \ITEE{#3}{TDobrowolski,JMogilski1990}{%%
      \BiB{#2}{T. Dobrowolski and J. Mogilski}%%
         {Problems on Topological Classification of Incomplete Metric Spaces}{Chapter 25 in:}%%
         {Open Problems in Topology}{J. van Mill and G.M. Reed (eds.), North-Holland Amsterdam}{1990}{411--429}{#1}}%%
   \ITEE{#3}{TDobrowolski,HTorunczyk1981}{%%
      \BIB{#2}{T. Dobrowolski and H. Toru\'{n}czyk}%%
         {Separable complete ANR's admitting a group structure are Hilbert manifolds}%%
         {\jRN{TopA}}{12}{1981}{229--235}{#1}}%%
   \ITEE{#3}{RGDouglas1966}{%%
      \BIB{#2}{R.G. Douglas}%%
         {On majorization, factorization and range inclusion of operators in Hilbert space}%%
         {\jRN{PAMS}}{17}{1966}{413--416}{#1}}%%
   \ITEE{#3}{CHDowker1947}{%%
      \BIB{#2}{C.H. Dowker}%%
         {Mapping theorems for non-compact spaces}%%
         {\jRN{AmJM}}{69}{1947}{200--242}{#1}}%%
   \ITEE{#3}{CHDowker1952}{%%
      \BIB{#2}{C.H. Dowker}%%
         {Topology of metric complexes}%%
         {\jRN{AmJM}}{74}{1952}{555--577}{#1}}%%
   \ITEE{#3}{JDugundji1951}{%%
      \BIB{#2}{J. Dugundji}%%
         {An extension of Tietze's theorem}%%
         {\jRN{PacJM}}{1}{1951}{353--367}{#1}}%%
   \ITEE{#3}{JDugundji1958}{%%
      \BIB{#2}{J. Dugundji}%%
         {Absolute neighborhood retracts and local connectedness for arbitrary metric spaces}%%
         {\jRN{ComposM}}{13}{1958}{229--246}{#1}}%%
   \ITEE{#3}{JDugundji1965}{%%
      \BIB{#2}{J. Dugundji}%%
         {Locally equiconnected spaces and absolute neighborhood retracts}%%
         {\jRN{FM}}{57}{1965}{187--193}{#1}}%%
   \ITEE{#3}{NDunford,JTSchwartz1958}{%%
      \BIb{#2}{N. Dunford and J.T. Schwartz}%%
         {Linear Operators, part I}%%
         {Interscience Publishers, New York}{1958}{#1}}%%
   \ITEE{#3}{NDunford,JTSchwartz1963}{%%
      \BIb{#2}{N. Dunford and J.T. Schwartz}%%
         {Linear Operators, part II}%%
         {Interscience Publishers, New York}{1963}{#1}}%%
   \ITEE{#3}{NDunford,JTSchwartz1971}{%%
      \BIb{#2}{N. Dunford and J.T. Schwartz}%%
         {Linear Operators, part III}%%
         {Wiley-Interscience, New York}{1971}{#1}}%%
   \ITEE{#3}{MLEaton,MDPerlman1977}{%%
      \BIB{#2}{M.L. Eaton and M.D. Perlman}%%
         {Reflection groups, generalized Schur functions and the geometry of majorization}%%
         {\jRN{AnnProb}}{5}{1977}{829--860}{#1}}%%
   \ITEE{#3}{EGEffros1965}{%%
      \BIB{#2}{E.G. Effros}%%
         {The Borel space of von Neumann algebras on a separable Hilbert space}%%
         {\jRN{PacJM}}{15}{1965}{1153--1164}{#1}}%%
   \ITEE{#3}{EGEffros1966}{%%
      \BIB{#2}{E.G. Effros}%%
         {Global structure in von Neumann algebras}%%
         {\jRN{TAMS}}{121}{1966}{434--454}{#1}}%%
   \ITEE{#3}{REspinola,MAKhamsi2001}{%%
      \BiB{#2}{R. Espinola and M.A. Khamsi}%%
         {Introduction to hyperconvex spaces}{Chapter XIII in:}{Handbook of Metric Fixed Point Theory}%%
         {W.A. Kirk and B. Sims (editors), Kluwer Academic Publishers}{2001}{391--435}{#1}}%%
   \ITEE{#3}{PAFillmore,JPWilliams1971}{%%
      \BIB{#2}{P.A. Fillmore and J.P. Williams}%%
         {On operator ranges}%%
         {\jRN{AdvM}}{7}{1971}{254--281}{#1}}%%
   \ITEE{#3}{JEells,NHKuiper1969}{%%
      \BIB{#2}{J. Eells and N.H. Kuiper}%%
         {Homotopy negligible subsets in infinite-dimensional manifolds}%%
         {\jRN{ComposM}}{21}{1969}{151--161}{#1}}%%
   \ITEE{#3}{REngelking1977}{%%
      \BIb{#2}{R. Engelking}%%
         {General Topology}%%
         {\jRN{PWN}}{1977}{#1}}%%
   \ITEE{#3}{REngelking1978}{%%
      \BIb{#2}{R. Engelking}%%
         {Dimension Theory}%%
         {\jRN{PWN}}{1978}{#1}}%%
   \ITEE{#3}{REngelking1989}{%%
      \BIb{#2}{R. Engelking}%%
         {General Topology. Revised and completed edition \textup{(Sigma series in pure mathematics, vol. 6)}}%%
         {Heldermann Verlag, Berlin}{1989}{#1}}%%
   \ITEE{#3}{PErdos,RDMauldin1976}{%%
      \BIB{#2}{P. Erd\"{o}s and R.D. Mauldin}%%
         {The nonexistence of certain invariant measures}%%
         {\jRN{PAMS}}{59}{1976}{321--322}{#1}}%%
   \ITEE{#3}{JErnest1976}{%%
      \BIB{#2}{J. Ernest}%%
         {Charting the operator terrain}%%
         {\jRN{MAMS}}{171}{1976}{207 pp}{#1}}%%
   \ITEE{#3}{RHFox1943}{%%
      \BIB{#2}{R.H. Fox}%%
         {On fiber spaces, II}%%
         {\jRN{BAMS}}{49}{1943}{733--735}{#1}}%%
   \ITEE{#3}{NAFriedman1970}{%%
      \BIb{#2}{N.A. Friedman}%%
         {Introduction to ergodic theory}%%
         {Van Nostrand Reinhold Company}{1970}{#1}}%%
   \ITEE{#3}{MFujii,MKajiwara,YKato,FKubo1976}{%%
      \BIB{#2}{M. Fujii, M. Kajiwara, Y. Kato, F. Kubo}%%
         {Decompositions of operators in Hilbert spaces}%%
         {\jRN{MathJap}}{21}{1976}{117--120}{#1}}%%
   \ITEE{#3}{SGao,ASKechris2003}{%%
      \BIB{#2}{S. Gao and A.S. Kechris}%%
         {On the classification of Polish metric spaces up to isometry}%%
         {\jRN{MAMS}}{161}{2003}{viii+78}{#1}}%%
   \ITEE{#3}{MIGarrido,FMontalvo1991}{%%
      \BIB{#2}{M.I. Garrido and F. Montalvo}%%
         {On some generalizations of the Kakutani-Stone and Stone-Weierstrass theorems}%%
         {\jRN{ExtrM}}{6}{1991}{156--159}{#1}}%%
   \ITEE{#3}{LGe,JShen2002}{%%
      \BIB{#2}{L. Ge and J. Shen}%%
         {Generator problem for certain property T factors}%%
         {\jRN{PNAS}}{99}{2002}{565--567}{#1}}%%
   \ITEE{#3}{IMGelfand,MANaimark1943}{%%
      \BIB{#2}{I.M. Gelfand and M.A. Naimark}%%
         {On the embedding of normed rings into the ring of operators in Hilbert space}%%
         {\jRN{MSb}}{12}{1943}{197--213}{#1}}%%
   \ITEE{#3}{FGesztesy,MMalamud,MMitrea,SNaboko2009}{%%
      \BIB{#2}{F. Gesztesy, M. Malamud, M. Mitrea, S. Naboko}%%
         {Generalized polar decompositions for closed operators in Hilbert spaces and some applications}%%
         {\jRN{IEOT}}{64}{2009}{83--113}{#1}}%%
   \ITEE{#3}{LGillman,MJerison1960}{%%
      \BIb{#2}{L. Gillman and M. Jerison}%%
         {Rings of continuous functions}%%
         {New York}{1960}{#1}}%%
   \ITEE{#3}{JGlimm1960}{%%
      \BIB{#2}{J. Glimm}%%
         {A Stone-Weierstrass theorem for $\CCc^*$-algebras}%%
         {\jRN{AnnM}}{72}{1960}{216--244}{#1}}%%
   \ITEE{#3}{GGodefroy,NJKalton2003}{%%
      \BIB{#2}{G. Godefroy and N.J. Kalton}%%
         {Lipschitz-free Banach spaces}%%
         {\jRN{SM}}{159}{2003}{121--141}{#1}}%%
   \ITEE{#3}{ICGohberg,MGKrein1967}{%%
      \BIB{#2}{I.C. Gohberg and M.G. Krein}%%
         {On a description of contraction operators similar to unitary ones}%%
         {\jRN{FunkAnalPril}}{1}{1967}{38--60}{#1}}%%
   \ITEE{#3}{ELGriffinJr1953}{%%
      \BIB{#2}{E.L. Griffin Jr.}%%
         {Some contributions to the theory of rings of operators}%%
         {\jRN{TAMS}}{75}{1953}{471--504}{#1}}%%
   \ITEE{#3}{ELGriffinJr1955}{%%
      \BIB{#2}{E.L. Griffin Jr.}%%
         {Some contributions to the theory of rings of operators II}%%
         {\jRN{TAMS}}{79}{1955}{389--400}{#1}}%%
   \ITEE{#3}{MGromov1981}{%%
      \BIB{#2}{M. Gromov}%%
         {Groups of polynomial growth and expanding maps}%%
         {\jRN{InHauEtSPM}}{53}{1981}{53--73}{#1}}%%
   \ITEE{#3}{MGromov1999}{%%
      \BIb{#2}{M. Gromov}%%
         {Metric Structures for Riemannian and Non-Riemannian Spaces}%%
         {Progress in Math. \textbf{152}, Birkh\"{a}user}{1999}{#1}}%%
   \ITEE{#3}{JDeGroot1956}{%%
      \BIB{#2}{J. de Groot}%%
         {Non-archimedean metrics in topology}%%
         {\jRN{PAMS}}{7}{1956}{948--953}{#1}}%%
   \ITEE{#3}{LCGrove,CTBenson1985}{%%
      \BIb{#2}{L.C. Grove and C.T. Benson}%%
         {Finite Reflection Group}%%
         {2nd ed., Springer-Verlag}{1985}{#1}}%%
   \ITEE{#3}{VIGurarii1966}{%%
      \BIB{#2}{V.I. Gurari\v{\i}}{Spaces of universal placement, isotropic spaces and a problem of Mazur %%
         on rotations of Banach spaces \textup{(Russian)}}%%
         {\jRN{SibirMZ}}{7}{1966}{1002--1013}{#1}}%%
   \ITEE{#3}{DWHadwin1976}{%%
      \BIB{#2}{D.W. Hadwin}%%
         {An operator-valued spectrum}%%
         {\jRN{NAMS}}{23}{1976}{A-163}{#1}}%%
   \ITEE{#3}{DWHadwin1977}{%%
      \BIB{#2}{D.W. Hadwin}%%
         {An operator-valued spectrum}%%
         {\jRN{IndianaUMJ}}{26}{1977}{329--340}{#1}}%%
   \ITEE{#3}{HHahn1932}{%%
      \BIb{#2}{H. Hahn}%%
         {Reelle Funktionen I}%%
         {Leipzig}{1932}{#1}}%%
   \ITEE{#3}{PRHalmos1950}{%%
      \BIb{#2}{P.R. Halmos}%%
         {Measure theory}%%
         {Van Nostrand, New York}{1950}{#1}}%%
   \ITEE{#3}{PRHalmos1951}{%%
      \BIb{#2}{P.R. Halmos}%%
         {Introduction to Hilbert Space and the Theory of Spectral Multiplicity}%%
         {Chelsea Publishing Company, New York}{1951}{#1}}%%
   \ITEE{#3}{PRHalmos1956}{%%
      \BIb{#2}{P.R. Halmos}%%
         {Lectures on Ergodic Theory}%%
         {Publ. Math. Soc. Japan, Tokyo}{1956}{#1}}%%
   \ITEE{#3}{PRHalmos1982}{%%
      \BIb{#2}{P.R. Halmos}%%
         {A Hilbert Space Problem Book}%%
         {Springer-Verlag New York Inc.}{1982}{#1}}%%
  \ITEE{#3}{PRHalmos,JEMcLaughlin1963}{%%
      \BIB{#2}{P.R. Halmos and J.E. McLaughlin}%%
         {Partial isometries}%%
         {\jRN{PacJM}}{13}{1963}{585--596}{#1}}%%
   \ITEE{#3}{RWHansell1972}{%%
      \BIB{#2}{R.W. Hansell}%%
         {On the nonseparable theory of Borel and Souslin sets}%%
         {\jRN{BAMS}}{78}{1972}{236--241}{#1}}%%
   \ITEE{#3}{FHausdorff1930}{%%
      \BIB{#2}{F. Hausdorff}%%
         {Erweiterung einer Hom\"{o}omorphie}%%
         {\jRN{FM}}{16}{1930}{353--360}{#1}}%%
   \ITEE{#3}{FHausdorff1934}{%%
      \BIB{#2}{F. Hausdorff}%%
         {\"{U}ber innere Abbildungen}%%
         {\jRN{FM}}{23}{1934}{279--291}{#1}}%%
   \ITEE{#3}{FHausdorff1938}{%%
      \BIB{#2}{F. Hausdorff}%%
         {Erweiterung einer stetigen Abbildung}%%
         {\jRN{FM}}{30}{1938}{40--47}{#1}}%%
   \ITEE{#3}{DWHenderson1971}{%%
      \BIB{#2}{D.W. Henderson}%%
         {Corrections and extensions of two papers about infinite-dimensional manifolds}%%
         {\jRN{GTopA}}{1}{1971}{321--327}{#1}}%%
   \ITEE{#3}{DWHenderson1975}{%%
      \BIB{#2}{D.W. Henderson}%%
         {$Z$-sets in ANR's}%%
         {\jRN{TAMS}}{213}{1975}{205--216}{#1}}%%
   \ITEE{#3}{DWHenderson,RMSchori1970}{%%
      \BIB{#2}{D.W. Henderson and R.M. Schori}%%
         {Topological classification of infinite-dimensional manifolds by homotopy type}%%
         {\jRN{BAMS}}{76}{1970}{121--124}{#1}}%%
   \ITEE{#3}{DWHenderson,JEWest1970}{%%
      \BIB{#2}{D.W. Henderson and J.E. West}%%
         {Triangulated infinite-dimensional manifolds}%%
         {\jRN{BAMS}}{76}{1970}{655--660}{#1}}%%
   \ITEE{#3}{BHoffmann1979}{%%
      \BIB{#2}{B. Hoffmann}%%
         {A compact contractible topological group is trivial}%%
         {\jRN{ArchM}}{32}{1979}{585--587}{#1}}%%
   \ITEE{#3}{DHofmann2002}{%%
      \BIB{#2}{D. Hofmann}%%
         {On a generalization of the Stone-Weierstrass theorem}%%
         {\jRN{ACS}}{10}{2002}{569--592}{#1}}%%
   \ITEE{#3}{GHognas,AMukherjea1995}{%%
      \BIb{#2}{G. H\"ogn\"as and A. Mukherjea}%%
         {Probability Measures on Semigroups. Convolution Products, Random Walks, and Random Matrices}%%
         {Plenum Press, New York}{1995}{#1}}%%
   \ITEE{#3}{MRHolmes1992}{%%
      \BIB{#2}{M.R. Holmes}%%
         {The universal separable metric space of Urysohn and isometric embeddings thereof in Banach spaces}%%
         {\jRN{FM}}{140}{1992}{199--223}{#1}}%%
   \ITEE{#3}{MRHolmes2008}{%%
      \BIB{#2}{M.R. Holmes}%%
         {The Urysohn space embeds in Banach spaces in just one way}%%
         {\jRN{TopA}}{155}{2008}{1479--1482}{#1}}%%
   \ITEE{#3}{RRHolmes,TYTam1999}{%%
      \BIB{#2}{R.R. Holmes and T.Y. Tam}%%
         {Distance to the convex hull of an orbit under the action of a compact group}%%
         {\jRN{JAusMSA}}{66}{1999}{331--357}{#1}}%%
   \ITEE{#3}{RHorn,RMathias1990}{%%
      \BIB{#2}{R. Horn and R. Mathias}%%
         {Cauchy-Schwartz inequalities associated with positive semidefinite matrices}%%
         {\jRN{LAA}}{142}{1990}{63--82}{#1}}%%
   \ITEE{#3}{GEHuhunaisvili1955}{%%
      \BIB{#2}{G.E. Huhunai\v{s}vili}%%
         {On a property of Urysohn's universal metric space}%%
         {\jRN{DANSSSR}}{101}{1955}{607--610 (Russian)}{#1}}%%
   \ITEE{#3}{JEHumphreys1990}{%%
      \BIb{#2}{J.E. Humphreys}%%
         {Reflection Groups and Coxeter Groups}%%
         {Cambridge University Press}{1990}{#1}}%%
   \ITEE{#3}{JRIsbell1964}{%%
      \BIB{#2}{J.R. Isbell}%%
         {Six theorems about injective metric spaces}%%
         {\jRN{CMHelv}}{39}{1964}{65--76}{#1}}%%
   \ITEE{#3}{SIzumino,YKato1985}{%%
      \BIB{#2}{S. Izumino and Y. Kato}%%
         {The closure of invertible operators on Hilbert space}%%
         {\jRN{ActaSM}}{49}{1985}{321--327}{#1}}%%
   \ITEE{#3}{CJiang2004}{%%
      \BIB{#2}{C. Jiang}%%
         {Similarity classification of Cowen-Douglas operators}%%
         {\jRN{CanadJM}}{56}{2004}{742--775}{#1}}%%
   \ITEE{#3}{WBJohnson,JLindenstrauss2001}{%%
      \BiB{#2}{W.B. Johnson and J. Lindenstrauss}{Basic Concepts in the Geometry of Banach Spaces}%%
         {Chapter 1 in:}{Handbook of the Geometry of Banach Spaces, Vol. 1}%%
         {W.B. Johnson and J. Lindenstrauss (editors), Elsevier Science B.V., Amsterdam}{2001}{1--84}{#1}}%%
   \ITEE{#3}{IBJung,JStochel2008}{%%
      \BIB{#2}{I.B. Jung and J. Stochel}%%
         {Subnormal operators whose adjoints have rich point spectrum}%%
         {\jRN{JFA}}{255}{2008}{1797--1816}{#1}}%%
   \ITEE{#3}{RVKadison,JRRingrose1983}{%%
      \BIb{#2}{R.V. Kadison and J.R. Ringrose}%%
         {Fundamentals of the Theory of Operator Algebras. Volume I: Elementary Theory}%%
         {Academic Press, Inc., New York-London}{1983}{#1}}%%
   \ITEE{#3}{RVKadison,JRRingrose1986}{%%
      \BIb{#2}{R.V. Kadison and J.R. Ringrose}%%
         {Fundamentals of the Theory of Operator Algebras. Volume II: Advanced Theory}%%
         {Academic Press, Inc., Orlando-London}{1986}{#1}}%%
   \ITEE{#3}{SKakutani1936}{%%
      \BIB{#2}{S. Kakutani}%%
         {\"{U}ber die Metrisation der topologischen Gruppen}%%
         {\jRN{ProcImpAcadTokyo}}{12}{1936}{82--84}{#1}}%%
   \ITEE{#3}{SKakutani1938}{%%
      \BIB{#2}{S. Kakutani}%%
         {Two fixed-point theorems concerning bicompact convex sets}%%
         {\jRN{ProcImpAcadTokyo}}{14}{1938}{242--245}{#1}}%%
   \ITEE{#3}{SKakutani1941}{%%
      \BIB{#2}{S. Kakutani}%%
         {Concrete representation of abstract L-spaces}%%
         {\jRN{AnnM}}{42}{1941}{523--537}{#1}}%%
   \ITEE{#3}{SKakutani1941a}{%%
      \BIB{#2}{S. Kakutani}%%
         {Concrete representation of abstract M-spaces}%%
         {\jRN{AnnM}}{42}{1941}{994--1024}{#1}}%%
   \ITEE{#3}{NKalton2007}{%%
      \BIB{#2}{N. Kalton}%%
         {Extending Lipschitz maps into $\CCc(K)$-spaces}%%
         {\jRN{IsraelJM}}{162}{2007}{275--315}{#1}}%%
   \ITEE{#3}{RKane2001}{%%
      \BIb{#2}{R. Kane}%%
         {Reflection Groups and Invariant Theory}%%
         {Canadian Mathematical Society, Springer}{2001}{#1}}%%
   \ITEE{#3}{VKannan,SRRaju1980}{%%
      \BIB{#2}{V. Kannan and S.R. Raju}%%
         {The nonexistence of invariant universal measures on semigroups}%%
         {\jRN{PAMS}}{78}{1980}{482--484}{#1}}%%
   \ITEE{#3}{IKaplansky1951}{%%
      \BIB{#2}{I. Kaplansky}%%
         {A theorem on rings of operators}%%
         {\jRN{PacJM}}{1}{1951}{227--232}{#1}}%%
   \ITEE{#3}{MKatetov1988}{%%
      \BiB{#2}{M. Kat\v{e}tov}{On universal metric spaces}{in: Frolik (ed.),}%%
         {General Topology and its Relations to Modern Analysis and Algebra VI. Proceedings of the Sixth Prague %%
         Topological Symposium 1986}{Heldermann Verlag Berlin}{1988}{323--330}{#1}}%%
   \ITEE{#3}{YKatznelson1960}{%%
      \BIB{#2}{Y. Katznelson}%%
         {Sur les alg\'{e}bres dont les \'{e}l\'{e}ments non n\'{e}gatifs admettent des racines carr\'{e}es}%%
         {\jRN{AnnSciEcNormSupT}}{77}{1960}{167--174}{#1}}%%
   \ITEE{#3}{OHKeller1931}{%%
      \BIB{#2}{O.H. Keller}%%
         {Die Homoiomorphie der kompakten konvexen Mengen in Hilbertschen Raum}%%
         {\jRN{MAnn}}{105}{1931}{748--758}{#1}}%%
   \ITEE{#3}{MAKhamsi,WAKirk,CMartinez2000}{%%
      \BIB{#2}{M.A. Khamsi, W.A. Kirk, C. Martinez}%%
         {Fixed point and selection theorems in hyperconvex spaces}%%
         {\jRN{PAMS}}{128}{2000}{3275--3283}{#1}}%%
   \ITEE{#3}{ABKhararazishvili1998}{%%
      \BIb{#2}{A.B. Khararazishvili}%%
         {Transformation groups and invariant measures. Set-theoretic aspects}%%
         {World Scientific Publishing Co., Inc., River Edge, NJ}{1998}{#1}}%%
   \ITEE{#3}{YKijima1987}{%%
      \BIB{#2}{Y. Kijima}%%
         {Fixed points of nonexpansive self-maps of a compact metric space}%%
         {\jRN{JMAnApp}}{123}{1987}{114--116}{#1}}%%
  \ITEE{#3}{JSKim,ChRKim,SGLee1980}{%%
      \BIB{#2}{J.S. Kim, Ch.R. Kim, S.G. Lee}%%
         {Reducing operator valued spectra of a Hilbert space operator}%%
         {\jRN{JKoreanMS}}{17}{1980}{123--129}{#1}}%%
   \ITEE{#3}{JKindler1995}{%%
      \BIB{#2}{J. Kindler}%%
         {Minimax theorems with applications to convex metric spaces}%%
         {\jRN{CollM}}{68}{1995}{179--186}{#1}}%%
   \ITEE{#3}{WAKirk1998}{%%
      \BIB{#2}{W.A. Kirk}%%
         {Hyperconvexity of $\RRR$-trees}%%
         {\jRN{FM}}{156}{1998}{67--72}{#1}}%%
   \ITEE{#3}{VLKleeJr1952}{%%
      \BIB{#2}{V.L. Klee Jr.}%%
         {Invariant metrics in groups (solution of a problem of Banach)}%%
         {\jRN{PAMS}}{3}{1952}{484--487}{#1}}%%
   \ITEE{#3}{HJKowalsky1957}{%%
      \BIB{#2}{H.J. Kowalsky}%%
         {Einbettung metrischer R\"{a}ume}%%
         {\jRN{ArchM}}{8}{1957}{336--339}{#1}}%%
   \ITEE{#3}{WKubis,MRubin2010}{%%
      \BIB{#2}{W. Kubi\'{s} and M. Rubin}%%
         {Extension and reconstruction theorems for the Urysohn universal metric space}%%
         {\jRN{CzMJ}}{60}{2010}{1--29}{#1}}%%
   \ITEE{#3}{KKuratowski1966}{%%
      \BIb{#2}{K. Kuratowski}%%
         {Topology. \textup{Vol. I}}%%
         {\jRN{PWN}}{1966}{#1}}%%
   \ITEE{#3}{KKuratowski,BKnaster1927}{%%
      \BIB{#2}{K. Kuratowski and B. Knaster}%%
         {A connected and connected im kleinen point set which contains no perfect subset}%%
         {\jRN{BAMS}}{33}{1927}{106--109}{#1}}%%
   \ITEE{#3}{KKuratowski,AMostowski1976}{%%
      \BIb{#2}{K. Kuratowski and A. Mostowski}%%
         {Set Theory with an Introduction to Descriptive Set Theory}%%
         {\jRN{PWN}}{1976}{#1}}%%
   \ITEE{#3}{GLewicki1992}{%%
      \BIB{#2}{G. Lewicki}%%
         {Bernstein's ``lethargy'' theorem in metrizable topological linear spaces}%%
         {\jRN{MonatM}}{113}{1992}{213--226}{#1}}%%
   \ITEE{#3}{ASLewis1996}{%%
      \BIB{#2}{A.S. Lewis}%%
         {Group invariance and convex matrix analysis}%%
         {\jRN{SIAMJMAA}}{17}{1996}{927--949}{#1}}%%
   \ITEE{#3}{C-KLi,N-KTsing1991}{%%
      \BIB{#2}{C.-K. Li and N.-K. Tsing}%%
         {$G$-invariant norms and $G(c)$-radii}%%
         {\jRN{LAA}}{150}{1991}{179--194}{#1}}%%
   \ITEE{#3}{AJLazar,JLindenstrauss1971}{%%
      \BIB{#2}{A.J. Lazar and J. Lindenstrauss}%%
         {Banach spaces whose duals are $L_1$ spaces and their representing matrices}%%
         {\jRN{ActaM}}{126}{1971}{165--193}{#1}}%%
   \ITEE{#3}{EHLieb,MLoss1997}{%%
      \BIb{#2}{E.H. Lieb and M. Loss}%%
         {Analysis \textup{(Graduate Studies in Mathematics, vol. 14)}}%%
         {Amer. Math. Soc., Providence, RI}{1997}{#1}}%%
   \ITEE{#3}{ALindenbaum1926}{%%
      \BIB{#2}{A. Lindenbaum}%%
         {Contributions \`{a} l'\'{e}tude de l'espace m\'{e}trique I}%%
         {\jRN{FM}}{8}{1926}{209--222}{#1}}%%
   \ITEE{#3}{DLindenstrauss,LTzafriri1971}{%%
      \BIB{#2}{D. Lindenstrauss and L. Tzafriri}%%
         {On the complemented subspaces problem}%%
         {\jRN{IsraelJM}}{9}{1971}{263--269}{#1}}%%
   \ITEE{#3}{RILoebl1986}{%%
      \BIB{#2}{R.I. Loebl}%%
         {A note on containment of operators}%%
         {\jRN{BAustrMS}}{33}{1986}{279--291}{#1}}%%
   \ITEE{#3}{LHLoomis1945}{%%
      \BIB{#2}{L.H. Loomis}%%
         {Abstract congruence and the uniqueness of Haar measure}%%
         {\jRN{AnnM}}{46}{1945}{348--355}{#1}}%%
   \ITEE{#3}{LHLoomis1949}{%%
      \BIB{#2}{L.H. Loomis}%%
         {Haar measure in uniform structures}%%
         {\jRN{DukeMJ}}{16}{1949}{193--208}{#1}}%%
   \ITEE{#3}{ERLorch1939}{%%
      \BIB{#2}{E.R. Lorch}%%
         {Bicontinuous linear transformation in certain vector spaces}%%
         {\jRN{BAMS}}{45}{1939}{564--569}{#1}}%%
   \ITEE{#3}{ATLundell,SWeingram1969}{%%
      \BIb{#2}{A.T. Lundell and S. Weingram}%%
         {The topology of CW-complexes}%%
         {Litton Educ. Publ.}{1969}{#1}}%%
   \ITEE{#3}{WLusky1976}{%%
      \BIB{#2}{W. Lusky}%%
         {The Gurarij spaces are unique}%%
         {\jRN{ArchM}}{27}{1976}{627--635}{#1}}%%
   \ITEE{#3}{WLusky1977}{%%
      \BIB{#2}{W. Lusky}%%
         {On separable Lindenstrauss spaces}%%
         {\jRN{JFA}}{26}{1977}{103--120}{#1}}%%
   \ITEE{#3}{DMaharam1942}{%%
      \BIB{#2}{D. Maharam}%%
         {On homogeneous measure algebras}%%
         {\jRN{PNatlUSA}}{28}{1942}{108--111}{#1}}%%
   \ITEE{#3}{MMalicki,SSolecki2009}{%%
      \BIB{#2}{M. Malicki and S. Solecki}%%
         {Isometry groups of separable metric spaces}%%
         {\jRN{MProcCambPhS}}{146}{2009}{67--81}{#1}}%%
   \ITEE{#3}{PMankiewicz1972}{%%
      \BIB{#2}{P. Mankiewicz}%%
         {On extension of isometries in normed linear spaces}%%
         {\jRN{BAPolSSSM}}{20}{1972}{367--371}{#1}}%%
   \ITEE{#3}{JMartinezMaurica,MTPellon1987}{%%
      \BIB{#2}{J. Martinez-Maurica and M.T. Pell\'{o}n}%%
         {Non-archimedean Chebyshev centers}%%
         {\jRN{IndagMP}}{90}{1987}{417--421}{#1}}%%
   \ITEE{#3}{KMaurin1980}{%%
      \BIb{#2}{K. Maurin}%%
         {Analysis, Part II}%%
         {D. Reidel, Dordrecht-Boston-London}{1980}{#1}}%%
   \ITEE{#3}{SMazur,SUlam1932}{%%
      \BIB{#2}{S. Mazur and S. Ulam}%%
         {Sur les transformationes isom\'{e}triques d'espaces vectoriels norm\'{e}s}%%
         {\jRN{CRASParis}}{194}{1932}{946--948}{#1}}%%
   \ITEE{#3}{SMazurkiewicz1920}{%%
      \BIB{#2}{S. Mazurkiewicz}%%
         {Sur les lignes de Jordan}%%
         {\jRN{FM}}{1}{1920}{166--209}{#1}}%%
   \ITEE{#3}{SMazurkiewicz,WSierpinski1920}{%%
      \BIB{#2}{S. Mazurkiewicz and W. Sierpi\'{n}ski}%%
         {Contributions a la topologie des ensembles denombrables}%%
         {\jRN{FM}}{1}{1920}{17--27}{#1}}%%
   \ITEE{#3}{MMbekhta1992}{%%
      \BIB{#2}{M. Mbekhta}%%
         {Sur la structure des composantes connexes semi-Fredholm de $B(H)$}%%
         {\jRN{PAMS}}{116}{1992}{521--524}{#1}}%%
   \ITEE{#3}{JEMcCarthy1996}{%%
      \BIB{#2}{J.E. McCarthy}%%
         {Boundary values and Cowen-Douglas curvature}%%
         {\jRN{JFA}}{137}{1996}{1--18}{#1}}%%
   \ITEE{#3}{JMelleray2007}{%%
      \BIB{#2}{J. Melleray}%%
         {Computing the complexity of the relation of isometry between separable Banach spaces}%%
         {\jRN{MLQ}}{53}{2007}{128--131}{#1}}%%
   \ITEE{#3}{JMelleray2007a}{%%
      \BIB{#2}{J. Melleray}%%
         {On the geometry of Urysohn's universal metric space}%%
         {\jRN{TopA}}{154}{2007}{384--403}{#1}}%%
   \ITEE{#3}{JMelleray2008}{%%
      \BIB{#2}{J. Melleray}%%
         {Some geometric and dynamical properties of the Urysohn space}%%
         {\jRN{TopA}}{155}{2008}{1531--1560}{#1}}%%
   \ITEE{#3}{JMelleray,FVPetrov,AMVershik2008}{%%
      \BIB{#2}{J. Melleray, F.V. Petrov, A.M. Vershik}%%
         {Linearly rigid metric spaces and the embedding problem}%%
         {\jRN{FM}}{199}{2008}{177--194}{#1}}%%
   \ITEE{#3}{EMichael1953}{%%
      \BIB{#2}{E. Michael}%%
         {Some extension theorems for continuous functions}%%
         {\jRN{PacJM}}{3}{1953}{789--806}{#1}}%%
   \ITEE{#3}{EMichael1954}{%%
      \BIB{#2}{E. Michael}%%
         {Local properties of topological spaces}%%
         {\jRN{DukeMJ}}{21}{1954}{163--171}{#1}}%%
   \ITEE{#3}{EMichael1956}{%%
      \BIB{#2}{E. Michael}%%
         {Selected selection theorems}%%
         {\jRN{AmMMon}}{58}{1956}{233--238}{#1}}%%
   \ITEE{#3}{EMichael1956a}{%%
      \BIB{#2}{E. Michael}%%
         {Continuous selections. I}%%
         {\jRN{AnnM}}{63}{1956}{361--382}{#1}}%%
   \ITEE{#3}{EMichael1956b}{%%
      \BIB{#2}{E. Michael}%%
         {Continuous selections. II}%%
         {\jRN{AnnM}}{64}{1956}{562--580}{#1}}%%
   \ITEE{#3}{EMichael1959}{%%
      \BIB{#2}{E. Michael}%%
         {A theorem on semi-continuous set-valued functions}%%
         {\jRN{DukeMJ}}{26}{1959}{647--652}{#1}}%%
   \ITEE{#3}{JVanMill1986}{%%
      \BIB{#2}{J. van Mill}%%
         {Another counterexample in ANR theory}%%
         {\jRN{PAMS}}{97}{1986}{136--138}{#1}}%%
   \ITEE{#3}{JVanMill2001}{%%
      \BIb{#2}{J. van Mill}%%
         {The Infinite-Dimensional Topology of Function Spaces %%
         \textup{(North-Holland Mathematical Library, vol. 64)}}%%
         {Elsevier, Amsterdam}{2001}{#1}}%%
   \ITEE{#3}{WMlak1991}{%%
      \BIb{#2}{W. Mlak}%%
         {Hilbert Spaces and Operator Theory}%%
         {PWN --- Polish Scientific Publishers and Kluwer Academic Publishers, Warszawa-Dordrecht}{1991}{#1}}%%
   \ITEE{#3}{JMogilski1979}{%%
      \BIB{#2}{J. Mogilski}%%
         {$CE$-decomposition of $l_2$-manifolds}%%
         {\jRN{BAPolSSSM}}{27}{1979}{309--314}{#1}}%%
   \ITEE{#3}{RLMoore1916}{%%
      \BIB{#2}{R.L. Moore}%%
         {On the foundations of plane analysis situs}%%
         {\jRN{TAMS}}{17}{1916}{131--164}{#1}}%%
   \ITEE{#3}{KMorita1955}{%%
      \BIB{#2}{K. Morita}%%
         {A condition for the metrizability of topological spaces and for $n$-dimensionality}%%
         {\jRN{SciRepTokyoA}}{5}{1955}{33--36}{#1}}%%
   \ITEE{#3}{AMukherjea,NATserpes1976}{%%
      \BIb{#2}{A. Mukherjea and N.A. Tserpes}%%
         {Measures on topological semigroups}%%
         {Springer Lecture Notes in Math. Vol. 547, Berlin}{1976}{#1}}%%
   \ITEE{#3}{JMycielski1974}{%%
      \BIB{#2}{J. Mycielski}%%
         {Remarks on invariant measures in metric spaces}%%
         {\jRN{CollM}}{32}{1974}{105--112}{#1}}%%
   \ITEE{#3}{SNNaboko1984}{%%
      \BIB{#2}{S.N. Naboko}%%
         {Conditions for similarity to unitary and selfadjoint operators}%%
         {\jRN{FunkAnalPril}}{18}{1984}{16--27}{#1}}%%
   \ITEE{#3}{LNachbin1965}{%%
      \BIb{#2}{L. Nachbin}%%
         {The Haar Integral}%%
         {D. Van Nostrand Company, Inc., Princeton-New Jersey-Toronto-New York-London}{1965}{#1}}%%
   \ITEE{#3}{TDNarang,SKGarg1991}{%%
      \BIB{#2}{T.D. Narang and S.K. Garg}%%
         {On the uniqueness of best approximation in non-archimedian spaces}%%
         {\jRN{PeriodMHung}}{22}{1991}{121--124}{#1}}%%
   \ITEE{#3}{JVonNeumann1930}{%%
      \BIB{#2}{J. von Neumann}%%
         {Zur Algebra der Funktionaloperationen und Theorie der normalen Operatoren}%%
         {\jRN{MAnn}}{102}{1930}{370--427}{#1}}%%
   \ITEE{#3}{JVonNeumann1934}{%%
      \BIB{#2}{J. von Neumann}%%
         {Zum Haarschen Mass in topologischen Gruppen}%%
         {\jRN{ComposM}}{1}{1934}{106--114}{#1}}%%
   \ITEE{#3}{JVonNeumann1937}{%%
      \BiB{#2}{J. von Neumann}%%
         {Some matrix-inequalities and metrization of matrix-space}{\jRN{TomskUnivRev}{} \textbf{1} (1937), 286--300; %%
         in }{Collected Works}{Pergamon, New York}{1962}{Vol. 4, 205--219}{#1}}%%
   \ITEE{#3}{JVonNeumann1949}{%%
      \BIB{#2}{J. von Neumann}%%
         {On Rings of Operators. Reduction Theory}%%
         {\jRN{AnnM}}{50}{1949}{401--485}{#1}}%%
   \ITEE{#3}{ONielson1973}{%%
      \BIB{#2}{O. Nielson}%%
         {Borel sets of von Neumann algebras}%%
         {\jRN{AmJM}}{95}{1973}{145--164}{#1}}%%
   \ITEE{#3}{pn2002}{\bibITEM{#2}{#1} \mypaplist{pn1}}%%
   \ITEE{#3}{pn2006a}{\bibITEM{#2}{#1} \mypaplist{pn2}}%%
   \ITEE{#3}{pn2006b}{\bibITEM{#2}{#1} \mypaplist{pn3}}%%
   \ITEE{#3}{pn2007}{\bibITEM{#2}{#1} \mypaplist{pn4}}%%
   \ITEE{#3}{pn2008a}{\bibITEM{#2}{#1} \mypaplist{pn5}}%%
   \ITEE{#3}{pn2008b}{\bibITEM{#2}{#1} \mypaplist{pn6}}%%
   \ITEE{#3}{pn2009a}{\bibITEM{#2}{#1} \mypaplist{pn7}}%%
   \ITEE{#3}{pn2009b}{\bibITEM{#2}{#1} \mypaplist{pn8}}%%
   \ITEE{#3}{pn2009c}{\bibITEM{#2}{#1} \mypaplist{pn9}}%%
   \ITEE{#3}{pn2010a}{\bibITEM{#2}{#1} \mypaplist{pn12}}%%
   \ITEE{#3}{pn2010b}{\bibITEM{#2}{#1} \mypaplist{pn13}}%%
   \ITEE{#3}{pn2011a}{\bibITEM{#2}{#1} \mypaplist{pn10}}%% G-represent
   \ITEE{#3}{pn2011b}{\bibITEM{#2}{#1} \mypaplist{pn15}}%% U-topol
   \ITEE{#3}{pn2011c}{\bibITEM{#2}{#1} \mypaplist{pn16}}%% note-inv
   \ITEE{#3}{pn2011d}{\bibITEM{#2}{#1} \mypaplist{pn17}}%% S-W
   \ITEE{#3}{pn2009x}{%% U-F2
      \bibITEM{#2}{#1} \mypaplist{pn11}}%% U-f
   \ITEE{#3}{pn2010x}{%%
      \bibITEM{#2}{#1} \mypaplist{pn14}}%% NTPVG
   \ITEE{#3}{pnXXXXb}{%% U-F2, U-F, SMF
      \bibITEM{#2}{#1} \mypaplist{pnX2}}%% functor
   \ITEE{#3}{pnXXXXc}{%% SPECTRUM
      \bibITEM{#2}{#1} \mypaplist{pnX3}}%% orbits
   \ITEE{#3}{pnXXXXd}{%% uvAg
      \bibITEM{#2}{#1} \mypaplist{pnX13}}%% note_ANR
   \ITEE{#3}{MNiezgoda1998}{%%
      \BIB{#2}{M. Niezgoda}%%
         {Group majorization and Schur type inequalities}%%
         {\jRN{LAA}}{268}{1998}{9--30}{#1}}%%
   \ITEE{#3}{MNiezgoda1998a}{%%
      \BIB{#2}{M. Niezgoda}%%
         {An analytical characterization of effective and of irreducible groups inducing cone orderings}%%
         {\jRN{LAA}}{269}{1998}{105--114}{#1}}%%
   \ITEE{#3}{MNiezgoda,TYTam2001}{%%
      \BIB{#2}{M. Niezgoda and T.Y. Tam}%%
         {On norm property of $G(c)$-radii and Eaton triples}%%
         {\jRN{LAA}}{336}{2001}{119--130}{#1}}%%
   \ITEE{#3}{APazy1983}{%%
      \BIb{#2}{A. Pazy}{Semigroups of Linear Operators %%
         and Applications to Partial Differential Equations \textup{(Applied Mathematical Sciences, vol. 44)}}%%
         {Springer-Verlag, New York}{1983}{#1}}%%
   \ITEE{#3}{APelc1982}{%%
      \BIB{#2}{A. Pelc}%%
         {Semiregular invariant measures on abelian groups}%%
         {\jRN{PAMS}}{86}{1982}{423--426}{#1}}%%
   \ITEE{#3}{RPenrose1955}{%%
      \BIB{#2}{R. Penrose}%%
         {A generalized inverse for matrices}%%
         {\jRN{ProcCambPhS}}{51}{1955}{406--413}{#1}}%%
   \ITEE{#3}{VPestov2006}{%%
      \BIb{#2}{V. Pestov}%%
         {Dynamics of infinite-dimensional groups. The Ramsey-Dvoretzky-Milman phenomenon}%%
         {University Lecture Series \textbf{40}, AMS, Providence, RI}{2006}{#1}}%%
   \ITEE{#3}{VPestov2007}{%%
      \BiB{#2}{V. Pestov}%%
         {Forty-plus annotated questions about large topological groups}%%
         {in:}{Open Problems in Topology II}{Elliot Pearl (editor), Elsevier B.V., Amsterdam}{2007}{439--450}{#1}}%%
   \ITEE{#3}{PVPetersen1993}{%%
      \BiB{#2}{P.V. Petersen}%%
         {Gromov-Hausdorff convergence of metric spaces}{in book:}{Differential Geometry: Riemannian Geometry %%
         (Los Angeles, CA, 1990)}{Amer. Math. Soc., Providence, RI}{1993}{489--504}{#1}}%%
   \ITEE{#3}{DRamachandran,MMisiurewicz1982}{%%
      \BIB{#2}{D. Ramachandran and M. Misiurewicz}%%
         {Hopf's theorem on invariant measures for a group of transformations}%%
         {\jRN{SM}}{74}{1982}{183--189}{#1}}%%
   \ITEE{#3}{JMRosenblatt1974}{%%
      \BIB{#2}{J.M. Rosenblatt}%%
         {Equivalent invariant measures}%%
         {\jRN{IsraelJM}}{17}{1974}{261--270}{#1}}%%
   \ITEE{#3}{HLRoyden1963}{%%
      \BIb{#2}{H.L. Royden}%%
         {Real Analysis}%%
         {The Macmillan Co., New York}{1963}{#1}}%%
   \ITEE{#3}{WRudin1962}{%%
      \BIb{#2}{W. Rudin}%%
         {Fourier Analysis on Groups \textup{(Interscience Tracts in Pure and Applied Mathematics, Number 12)}}%%
         {Interscience Publishers, New York}{1962}{#1}}%%
   \ITEE{#3}{WRudin1991}{%%
      \BIb{#2}{W. Rudin}%%
         {Functional Analysis}%%
         {McGraw-Hill Science}{1991}{#1}}%%
   \ITEE{#3}{TSaito1972}{%%
      \BiB{#2}{T. Sait\^{o}}{Generations of von Neumann algebras}%%
         {Lecture Notes in Math. vol. 247}{\textup{(}Lecture on Operator Algebras\textup{)}}%%
         {Springer, Berlin-Heidelberg-New York}{1972}{435--531}{#1}}%%
   \ITEE{#3}{KSakai,MYaguchi2003}{%%
      \BIB{#2}{K. Sakai and M. Yaguchi}%%
         {Characterizing manifolds modeled on certain dense subspaces of non-separable Hilbert spaces}%%
         {\jRN{TsukubaJM}}{27}{2003}{143--159}{#1}}%%
   \ITEE{#3}{SSakai1971}{%%
      \BIb{#2}{S. Sakai}%%
         {$\CCc^*$-Algebras and $\WWw^*$-Algebras}%%
         {Springer-Verlag, Berlin-Heidelberg-New York}{1971}{#1}}%%
   \ITEE{#3}{RSchori1971}{%%
      \BIB{#2}{R. Schori}%%
         {Topological stability for infinite-dimensional manifolds}%%
         {\jRN{ComposM}}{23}{1971}{87--100}{#1}}%%
   \ITEE{#3}{JTSchwartz1967}{%%
      \BIb{#2}{J.T. Schwartz}%%
         {$\WWw^*$-algebras}%%
         {Gordon and Breach, Science Publishers Inc., New York-London-Paris}{1967}{#1}}%%
   \ITEE{#3}{ZSemadeni1971}{%%
      \BIb{#2}{Z. Semadeni}%%
         {Banach Spaces of Continuous Functions (Vol. I)}%%
         {\jRN{PWN}}{1971}{#1}}%%
   \ITEE{#3}{JPSerre1951}{%%
      \BIB{#2}{J.-P. Serre}%%
         {Homologie singuli\`{e}re des espaces fibr\'{e}s}%%
         {\jRN{AnnM}}{54}{1951}{425--505}{#1}}%%
   \ITEE{#3}{DSherman2007}{%%
      \BIB{#2}{D. Sherman}%%
         {On the dimension theory of von Neumann algebras}%%
         {\jRN{MScand}}{101}{2007}{123--147}{#1}}%%
   \ITEE{#3}{WSierpinski1928}{%%
      \BIB{#2}{W. Sierpi\'{n}ski}%%
         {Sur les projections des ensembles compl\'{e}mentaires aux ensembles \textup{(A)}}%%
         {\jRN{FM}}{11}{1928}{117--122}{#1}}%%
   \ITEE{#3}{MSlocinski1980}{%%
      \BIB{#2}{M. S\l{}oci\'{n}ski}%%
         {On the Wold-type decomposition of a pair of commuting isometries}%%
         {\jRN{APM}}{37}{1980}{255--262}{#1}}%%
   \ITEE{#3}{RCSteinlage1975}{%%
      \BIB{#2}{R.C. Steinlage}%%
         {On Haar measure in locally compact $T_2$ spaces}%%
         {\jRN{AmJM}}{97}{1975}{291--307}{#1}}%%
   \ITEE{#3}{JStochel,FHSzafraniec1989}{%%
      \BIB{#2}{J. Stochel and F.H. Szafraniec}%%
         {On normal extensions of unbounded operators. III. Spectral properties}%%
         {\jRN{PublRIMSKyoto}}{25}{1989}{105--139}{#1}}%%
   \ITEE{#3}{JStochel,FHSzafraniec1989a}{%%
      \BIB{#2}{J. Stochel and F.H. Szafraniec}%%
         {The normal part of an unbounded operator}%%
         {\jRN{ProcKonink}}{92}{1989}{495--503}{#1}}%%
   \ITEE{#3}{AHStone1962}{%%
      \BIB{#2}{A.H. Stone}%%
         {Absolute $\FFf_{\sigma}$-spaces}%%
         {\jRN{PAMS}}{13}{1962}{495--499}{#1}}%%
   \ITEE{#3}{AHStone1962a}{%%
      \BIB{#2}{A.H. Stone}%%
         {Non-separable Borel sets}%%
         {\jRN{DissM}}{28}{1962}{41 pages}{#1}}%%
   \ITEE{#3}{AHStone1972}{%%
      \BIB{#2}{A.H. Stone}%%
         {Non-separable Borel sets II}%%
         {\jRN{GTopA}}{2}{1972}{249--270}{#1}}%%
   \ITEE{#3}{MHStone1937}{%%
      \BIB{#2}{M.H. Stone}%%
         {Application of the theory of Boolean rings to general topology}%%
         {\jRN{TAMS}}{41}{1937}{375--481}{#1}}%%
   \ITEE{#3}{MHStone1948}{%%
      \BIB{#2}{M.H. Stone}%%
         {The generalized Weierstrass approximation theorem}%%
         {\jRN{MMag}}{21}{1948}{167--184}{#1}}%%
   \ITEE{#3}{BSz-Nagy1947}{%%
      \BIB{#2}{B. Sz.-Nagy}%%
         {On uniformly bounded linear transformations in Hilbert space}%%
         {\jRN{ActaSM}}{11}{1947}{152--157}{#1}}%%
   \ITEE{#3}{WTakahashi1970}{%%
      \BIB{#2}{W. Takahashi}%%
         {A convexity in metric space and nonexpansive mappings, I}%%
         {\jRN{KodaiMSemRep}}{22}{1970}{142--149}{#1}}%%
   \ITEE{#3}{MTakesaki2002}{%%
      \BIb{#2}{M. Takesaki}%%
         {Theory of Operator Algebras I \textup{(Encyclopaedia of Mathematical Sciences, Volume 124)}}%%
         {Springer-Verlag, Berlin-Heidelberg-New York}{2002}{#1}}%%
   \ITEE{#3}{MTakesaki2003}{%%
      \BIb{#2}{M. Takesaki}%%
         {Theory of Operator Algebras II \textup{(Encyclopaedia of Mathematical Sciences, Volume 125)}}%%
         {Springer-Verlag, Berlin-Heidelberg-New York}{2003}{#1}}%%
   \ITEE{#3}{MTakesaki2003a}{%%
      \BIb{#2}{M. Takesaki}%%
         {Theory of Operator Algebras III \textup{(Encyclopaedia of Mathematical Sciences, Volume 127)}}%%
         {Springer-Verlag, Berlin-Heidelberg-New York}{2003}{#1}}%%
   \ITEE{#3}{TYTam1999}{%%
      \BIB{#2}{T.Y. Tam}%%
         {An extension of a result of Lewis}%%
         {\jRN{ELA}}{5}{1999}{1--10}{#1}}%%
   \ITEE{#3}{TYTam2000}{%%
      \BIB{#2}{T.Y. Tam}%%
         {Group majorization, Eaton triples and numerical range}%%
         {\jRN{LMLA}}{47}{2000}{11--28}{#1}}%%
   \ITEE{#3}{TYTam2002}{%%
      \BIB{#2}{T.Y. Tam}%%
         {Generalized Schur-concave functions and Eaton triples}%%
         {\jRN{LMLA}}{50}{2002}{113--120}{#1}}%%
   \ITEE{#3}{TYTam,WCHill2001}{%%
      \BIB{#2}{T.Y. Tam and W.C. Hill}%%
         {On $G$-invariant norms}%%
         {\jRN{LAA}}{331}{2001}{101--112}{#1}}%%
   \ITEE{#3}{AFTiman,IAVestfrid1983}{%%
      \BIB{#2}{A.F. Timan and I.A. Vestfrid}%%
         {Any separable ultrametric space can be isometrically imbedded in $l_2$}%%
         {\jRN{FAA}}{17}{1983}{70--71}{#1}}%%
   \ITEE{#3}{JTomiyama1958}{%%
      \BIB{#2}{J. Tomiyama}%%
         {Generalized dimension function for $\WWw^*$-algebras of infinite type}%%
         {\jRN{TohokuMJ} (2)}{10}{1958}{121--129}{#1}}%%
   \ITEE{#3}{HTorunczyk1970}{%%
      \BIB{#2}{H. Toru\'{n}czyk}%%
         {Remarks on Anderson's paper ``On topological infinite deficiency''}%%
         {\jRN{FM}}{66}{1970}{393--401}{#1}}%%
   \ITEE{#3}{HTorunczyk1970a}{%%
      \BIb{#2}{H. Toru\'{n}czyk}%%
         {$G$-$K$-absorbing and skeletonized sets in metric spaces}%%
         {Ph.D. thesis, Inst. Math. Polish Acad. Sci., Warszawa}{1970}{#1}}%%
   \ITEE{#3}{HTorunczyk1972}{%%
      \BIB{#2}{H. Toru\'{n}czyk}%%
         {A short proof of Hausdorff's theorem on extending metrics}%%
         {\jRN{FM}}{77}{1972}{191--193}{#1}}%%
   \ITEE{#3}{HTorunczyk1974}{%%
      \BIB{#2}{H. Toru\'{n}czyk}%%
         {Absolute retracts as factors of normed linear spaces}%%
         {\jRN{FM}}{86}{1974}{53--67}{#1}}%%
   \ITEE{#3}{HTorunczyk1975}{%%
      \BIB{#2}{H. Toru\'{n}czyk}%%
         {On Cartesian factors and the topological classification of linear metric spaces}%%
         {\jRN{FM}}{88}{1975}{71--86}{#1}}%%
   \ITEE{#3}{HTorunczyk1978}{%%
      \BIB{#2}{H. Toru\'{n}czyk}%%
         {Concerning locally homotopy negligible sets and characterization of $l_2$-manifolds}%%
         {\jRN{FM}}{101}{1978}{93--110}{#1}}%%
   \ITEE{#3}{HTorunczyk1980}{%%
      \BiB{#2}{H. Toru\'{n}czyk}{Characterization of infinite-dimensional manifolds}{in:}%%
         {Proceedings of the International Conference on Geometric Topology (Warsaw, 1978)}%%
         {\jRN{PWN}}{1980}{431--437}{#1}}%%
   \ITEE{#3}{HTorunczyk1981}{%%
      \BIB{#2}{H. Toru\'{n}czyk}%%
         {Characterizing Hilbert space topology}%%
         {\jRN{FM}}{111}{1981}{247--262}{#1}}%%
   \ITEE{#3}{HTorunczyk1985}{%%
      \BIB{#2}{H. Toru\'{n}czyk}%%
         {A correction of two papers concerning Hilbert manifolds}%%
         {\jRN{FM}}{125}{1985}{89--93}{#1}}%%
   \ITEE{#3}{KTsuda1985}{%%
      \BIB{#2}{K. Tsuda}%%
         {A note on closed embeddings of finite dimensional metric spaces}%%
         {\jRN{BLondMS}}{17}{1985}{273--278}{#1}}%%
   \ITEE{#3}{PSUrysohn1925}{%%
      \BIB{#2}{P.S. Urysohn}%%
         {Sur un espace m\'{e}trique universel}%%
         {\jRN{CRASParis}}{180}{1925}{803--806}{#1}}%%
   \ITEE{#3}{PSUrysohn1927}{%%
      \BIB{#2}{P.S. Urysohn}%%
         {Sur un espace m\'{e}trique universel}%%
         {\jRN{BullSM}}{51}{1927}{43--64, 74--96}{#1}}%%
   \ITEE{#3}{VVUspenskij1986}{%%
      \BIB{#2}{V.V. Uspenskij}%%
         {A universal topological group with a countable basis}%%
         {\jRN{FAA}}{20}{1986}{86--87}{#1}}%%
   \ITEE{#3}{VVUspenskij1990}{%%
      \BIB{#2}{V.V. Uspenskij}%%
         {On the group of isometries of the Urysohn universal metric space}%%
         {\jRN{CMUC}}{31}{1990}{181--182}{#1}}%%
   \ITEE{#3}{VVUspenskij2004}{%%
      \BIB{#2}{V.V. Uspenskij}%%
         {The Urysohn universal metric space is homeomorphic to a Hilbert space}%%
         {\jRN{TopA}}{139}{2004}{145--149}{#1}}%%
   \ITEE{#3}{VVUspenskij2008}{%%
      \BIB{#2}{V.V. Uspenskij}%%
         {On subgroups of minimal topological groups}%%
         {\jRN{TopA}}{155}{2008}{1580--1606}{#1}}%%
   \ITEE{#3}{VSVaradarajan1963}{%%
      \BIB{#2}{V.S. Varadarajan}%%
         {Groups of automorphisms of Borel spaces}%%
         {\jRN{TAMS}}{109}{1963}{191--220}{#1}}%%
   \ITEE{#3}{AMVershik1998}{%%
      \BIB{#2}{A.M. Vershik}%%
         {The universal Urysohn space, Gromov's metric triples, and random metrics on the series of natural numbers}%%
         {\jRN{UspekhiMN}}{53}{1998}{57--64}{#1} English translation: \jRN{RussMS}{} \textbf{53} (1998), 921--928. %%
         Correction: \jRN{UspekhiMN}{} \textbf{56} (2001), p. 207. English translation: \jRN{RussMS}{} \textbf{56} %%
         (2001), p. 1015.}%%
   \ITEE{#3}{AMVershik2002}{%%
      \BIb{#2}{A.M. Vershik}%%
         {Random metric spaces and the universal Urysohn space}%%
         {Fundamental Mathematics Today. 10th anniversary of the Independent Moscow University. MCCME Publ.}{2002}{#1}}%%
   \ITEE{#3}{NWeaver1999}{%%
      \BIb{#2}{N. Weaver}%%
         {Lipschitz Algebras}%%
         {World Scientific}{1999}{#1}}%%
   \ITEE{#3}{JWeidmann1980}{%%
      \BIb{#2}{J. Weidmann}%%
         {Linear Operators in Hilbert Spaces}%%
         {(Graduate Texts in Mathematics, vol. 68) Springer-Verlag New York Inc.}{1980}{#1}}%%
   \ITEE{#3}{JEWest1969}{%%
      \BIB{#2}{J.E. West}%%
         {Approximating homotopies by isotopies in Fr\'{e}chet manifolds}%%
         {\jRN{BAMS}}{75}{1969}{1254--1257}{#1}}%%
   \ITEE{#3}{JEWest1969a}{%%
      \BIB{#2}{J.E. West}%%
         {Fixed-point sets of transformation groups on infinite-product spaces}%%
         {\jRN{PAMS}}{21}{1969}{575--582}{#1}}%%
   \ITEE{#3}{JEWest1970}{%%
      \BIB{#2}{J.E. West}%%
         {The ambient homeomorphy of infinite-dimensional Hilbert spaces}%%
         {\jRN{PacJM}}{34}{1970}{257--267}{#1}}%%
   \ITEE{#3}{JHCWhitehead1949}{%%
      \BIB{#2}{J.H.C. Whitehead}%%
         {Combinatorial homotopy I}%%
         {\jRN{BAMS}}{55}{1949}{213--245}{#1}}%%
   \ITEE{#3}{GTWhyburn1942}{%%
      \BIb{#2}{G. T. Whyburn}%%
         {Analytic Topology}%%
         {Amer. Math. Soc. Colloquium Publications (vol. XXVIII), New York}{1942}{#1}}%%
   \ITEE{#3}{WWogen1969}{%%
      \BIB{#2}{W. Wogen}%%
         {On generators for von Neumann algebras}%%
         {\jRN{BAMS}}{75}{1969}{95--99}{#1}}%%
   \ITEE{#3}{RYTWong1967}{%%
      \BIB{#2}{R.Y.T. Wong}%%
         {On homeomorphisms of certain infinite dimensional spaces}%%
         {\jRN{TAMS}}{128}{1967}{148--154}{#1}}%%
   \ITEE{#3}{LYang,JZhang1987}{%%
      \BIB{#2}{L. Yang and J. Zhang}%%
         {Average distance constants of some compact convex space}%%
         {\jRN{JChinUST}}{17}{1987}{17--23}{#1}}%%
   \ITEE{#3}{PZakrzewski1993}{%%
      \BIB{#2}{P. Zakrzewski}%%
         {The existence of invariant $\sigma$-finite measures for a group of transformations}%%
         {\jRN{IsraelJM}}{83}{1993}{275--287}{#1}}%%
   \ITEE{#3}{PZakrzewski2002}{%%
      \BIb{#2}{P. Zakrzewski}%%
         {Measures on Algebraic-Topological Structures, Handbook of Measure Thoery}%%
         {E. Pap, ed., Elsevier, Amsterdam}{2002, 1091--1130}{#1}}%%
   \ITEE{#3}{KZhu2000}{%%
      \BIB{#2}{K. Zhu}%%
         {Operators in Cowen-Douglas classes}%%
         {\jRN{IllinoisJM}}{44}{2000}{767--783}{#1}}%%
   }

\newcommand{\mypaplist}[2][]{%%
   \ITEE{#2}{pn1}{%%
      \myBIB{Separate and joint similarity to families of normal operators}%%
         {\jRN[#1]{SM}}{149}{2002}{39--62}}%%
   \ITEE{#2}{pn2}{%%
      \myBIB{Locally arcwise connected metrizable spaces with the fixed point property are complete-metrizable}%%
         {\jRN[#1]{TopA}}{153}{2006}{1639--1642}}%%
   \ITEE{#2}{pn3}{%%
      \myBIB{Invariant measures for equicontinuous semigroups of continuous transformations of a compact Hausdorff space}%%
         {\jRN[#1]{TopA}}{153}{2006}{3373--3382}}%%
   \ITEE{#2}{pn4}{%%
      \myBIB{Approximation of the Hausdorff distance by the distance of continuous surjections}%%
         {\jRN[#1]{TopA}}{154}{2007}{655--664}}%%
   \ITEE{#2}{pn5}{%%
      \myBIB{Generalized Haar integral}%%
         {\jRN[#1]{TopA}}{155}{2008}{1323--1328}}%%
   \ITEE{#2}{pn6}{%%
      \myBIB{Integration and Lipschitz functions}%%
         {\jRN[#1]{RCMP}}{57}{2008}{391--399}}%%
   \ITEE{#2}{pn7}{%%
      \myBIB{Canonical Banach function spaces generated by Urysohn universal spaces. Measures as Lipschitz maps}%%
         {\jRN[#1]{SM}}{192}{2009}{97--110}}%%
   \ITEE{#2}{pn8}{%%
      \myBIB{Urysohn universal spaces as metric groups of exponent $2$}%%
         {\jRN[#1]{FM}}{204}{2009}{1--6}}%%
   \ITEE{#2}{pn9}{%%
      \myBIB{Central subsets of Urysohn universal spaces}%%
         {\jRN[#1]{CMUC}}{50}{2009}{445--461}}%%
   \ITEE{#2}{pn10}{%%
      \myBIB[P. Niemiec and T.Y. Tam]{A representation of $G$-in\-variant norms for Eaton triple}%%
         {\jRN[#1]{JCA}}{18}{2011}{59--65}}%%
   \ITEE{#2}{pn11}{%% U-F2
      \myBIB{Functor of extension of contractions on Urysohn universal spaces}%%
         {\jRN[#1]{ACS}}{}{2009}{\texttt{DOI: 10.1007/s10485-009-9218-z}}}%%
   \ITEE{#2}{pn12}{%%
      \myBIB{Ultra-$\mM$-separability}%%
         {\jRN[#1]{TopA}}{157}{2010}{669--673}}%%
   \ITEE{#2}{pn13}{%%
      \myBIB{Functor of extension of $\Lambda$-isometric maps between central subsets %%
         of the unbounded Urysohn universal space}{\jRN[#1]{CMUC}}{51}{2010}{541--549}}%%
   \ITEE{#2}{pn14}{%%
      \myBIB{Normed topological pseudovector groups}{\jRN[#1]{ACS}}{}{2010}%%
         {\ITE{\equal{#1}{}}{\texttt{DOI: 10.1007/s10485\-010-9239-7}}{\texttt{DOI: 10.1007/s10485-010-9239-7}}}}%%
   \ITEE{#2}{pn15}{%%
      \myBIB{Topological structure of Urysohn universal spaces}%%
         {\jRN[#1]{TopA}}{158}{2011}{352--359}}%%
   \ITEE{#2}{pn16}{%%
      \myBIB{A note on invariant measures}%%
         {\jRN[#1]{OpusM}}{31}{2011}{425--431}}%%
   \ITEE{#2}{pn17}{%%
      \myBIB{Strengthened Stone-Weierstrass type theorem}%%
         {\jRN[#1]{OpusM}}{31}{2011}{645--650}}%%
   \ITEE{#2}{pnX2}{%% U-F2, U-F, SMF
      \myBAPP{Functor of continuation in Hilbert cube and Hilbert space}%%
         {to appear in \jRN[#1]{FM}}}%%
   \ITEE{#2}{pnX3}{%% SPECTRUM
      \myBAPP{Norm closures of orbits of bounded operators}%%
         {to appear.}}%%
   \ITEE{#2}{pnX6}{%%
      \myBAPP{Extending maps by injective $\sigma$-$Z$-maps in Hilbert manifolds}%%
         {to appear in \jRN[#1]{BullPol}}}%%
   \ITEE{#2}{pnX7}{%%
      \myBAPP{Spaces of measurable functions}%%
         {submitted to \jRN[#1]{CollectM}}}%%
   \ITEE{#2}{pnX8}{%%
      \myBAPP{Normal systems over ANR's, rigid embeddings and nonseparable absorbing sets}%%
         {submitted to \jRN[#1]{ActaMSinES}}}%%
   \ITEE{#2}{pnX9}{%%
      \myBAPP{Borel structure of the spectrum of a closed operator}%%
         {submitted to \jRN[#1]{SM}}}%%
   \ITEE{#2}{pnX10}{%%
      \myBAPP{Central points and measures and dense subsets of compact metric spaces}%%
         {submitted to \jRN[#1]{TopMethNA}}}%%
   \ITEE{#2}{pnX11}{%%
      \myBAPP{Generalized absolute values and polar decompositions of a bounded operator}%%
         {submitted to \jRN[#1]{IEOT}.}}%%
   \ITEE{#2}{pnX12}{%%
      \myBAPP{Ultrametrics, extending of Lipschitz maps and nonexpansive selections}%%
         {accepted for publication in \jRN[#1]{HJM}}}%%
   \ITEE{#2}{pnX13}{%%
      \myBAPP{A note on ANR's}%%
         {submitted to \jRN[#1]{TopA}}}%%
   \ITEE{#2}{pnX14}{%%
      \myBAPP{Problem with almost everywhere equality}%%
         {submitted to \jRN[#1]{ArchM}}}%%
   \ITEE{#2}{pnX15}{%%
      \myBAPP{Universal valued Abelian groups}%%
         {submitted to \jRN[#1]{LNM}}}%%
   \ITEE{#2}{pnX16}{%%
      \myBAPP{Unitary equivalence and decompositions of finite systems of closed densely defined operators %%
         in Hilbert spaces}{submitted to \jRN[#1]{DissM}}}%%
   }

\begin{document}

\title[Problem with almost everywhere equality]{Problem with almost\\everywhere equality}
\myData
\begin{abstract}
A topological space $Y$ is said to have (AEEP) if the following condition is fulfilled. Whenever $(X,\Mm)$ is
a measurable space and $f,g\dd X \to Y$ are two measurable functions, then the set $\Delta(f,g) = \{x \in X\dd\
f(x) = g(x)\}$ is a member of $\Mm$. It is shown that a metrizable space $Y$ has (AEEP) iff the cardinality
of $Y$ is no greater than $2^{\aleph_0}$.\\
\textit{2010 MSC: Primary 28A20; Secondary 28A05.}\\
Key words: measurability; almost everywhere equality; metrizable space; measurable space; measurable function; measurable
set.
\end{abstract}
\maketitle

%%\thisPaper{pn??}

\SECT{Introduction}

In several aspects of mathematics dealing with measurability (such as measure theory, descriptive set theory, stochastic
processes, ergodic theory, study of $L^p$-spaces, etc.) the idea of identifying functions which are equal almost everywhere
is quite natural. The experience gained from real-valued functions may lead to an oversight that the set on which two
measurable functions (taking values in an arbitrary topological space) coincide is always measurable. It is well-known and
quite easy to prove that this happens when functions take values in a space with countable base or, if they have separable
images lying in a metrizable space. However, it is not true in general. The reason for this is that $\Bb(Y) \otimes \Bb(Y)$
differs (in general) from $\Bb(Y \times Y)$ where $\Bb(Y)$ is the $\sigma$-algebra of all Borel subsets of a topological
space $Y$. Let us say that a topological space $Y$ has \textit{almost everywhere equality property} (briefly, (AEEP))
if $Y$ satisfies the condition stated in Abstract. It is easily seen (see \LEM{delta} below) that $Y$ has (AEEP) iff
the diagonal of $Y$ belongs to $\Bb(Y) \otimes \Bb(Y)$. So, it may turn out that $Y$ has (AEEP) but still $\Bb(Y) \otimes
\Bb(Y) \neq \Bb(Y \times Y)$. The aim of this short note is to prove that a metrizable space has (AEEP) iff $\card(Y)
\leqsl 2^{\aleph_0}$. Thus, among metrizable spaces only those whose topological weight is no greater than $2^{\aleph_0}$
have (AEEP). It may seem surprising that not only separable spaces appear in this characterization.\par
Nonseparable metric spaces are widely investigated in functional analysis and operator theory. In fact, the Banach algebra
of all bounded linear operators acting on a separable Banach space is usually nonseparable. Also infinite-dimensional
von Neumann algebras are nonseparable. A special and a very important in theory of geometry of Banach spaces example
of them is the Banach space $l^{\infty}$ of all bounded sequences. So, our result may find applications in investigations
of these spaces.\par
\textbf{Notation and terminology.} For every topological space $Y$, $\Bb(Y)$ stands for the $\sigma$-algebra of all Borel
subsets of $Y$. That is, $\Bb(Y)$ is the smallest $\sigma$-algebra containing all open sets. Whenever $(\Omega,\Mm)$ and
$(\Lambda,\Nn)$ are two measurable spaces, a function $f\dd (\Omega,\Mm) \to (\Lambda,\Nn)$ is said to be
\textit{measurable} iff $f^{-1}(B) \in \Mm$ for each $B \in \Nn$. $\Mm \otimes \Nn$ denotes the product $\sigma$-algebra
of $\Mm$ and $\Nn$, i.e. $\Mm \otimes \Nn$ is the smallest $\sigma$-algebra on $\Omega \times \Lambda$ which contains all
sets of the form $A \times B$ with $A \in \Mm$ and $B \in \Nn$. If $g\dd (\Omega,\Mm) \to Y$ where $Y$ is a topological
space, $g$ is measurable if $g^{-1}(U) \in \Mm$ for any open set $U \subset Y$ or, equivalently, if $g\dd (\Omega,\Mm) \to
(Y,\Bb(Y))$ is measurable. For two functions $u,v \dd D \to E$, $\Delta(f,g)$ stands for the set $\{x \in D\dd\ u(x) =
v(x)\}$. Additionally, for every set $E$, $\Delta_E$ denotes the diagonal of $E$, i.e. $\Delta_E = \{(x,x)\dd\ x \in E\}$;
and $\card(E)$ is the cardinality of $E$.

\SECT{The result}

We begin with a simple

\begin{lem}{delta}
For a topological space $Y$ \tfcae
\begin{enumerate}[\upshape(i)]
\item $Y$ has \textup{(AEEP)},
\item $\Delta_Y \in \Bb(Y) \otimes \Bb(Y)$.
\end{enumerate}
\end{lem}
\begin{proof}
If $f,g\dd (\Omega,\Mm) \to Y$ are two measurable functions, then the function $h\dd (\Omega,\Mm) \ni \omega \mapsto
(f(\omega),g(\omega)) \in (Y \times Y,\Bb(Y) \otimes \Bb(Y))$ is measurable as well and thus $\Delta(f,g) =
h^{-1}(\Delta_Y)$ is a member of $\Mm$. This shows that (i) follows from (ii). To see the converse, notice that
the natural projections $p_j\dd (Y \times Y,\Bb(Y) \otimes \Bb(Y)) \ni (y_1,y_2) \mapsto y_j \in Y$ ($j=1,2$) are
measurable and that $\Delta(p_1,p_2) = \Delta_Y$ which finishes the proof.
\end{proof}

\begin{lem}{continuum}
For an arbitrary topological space $Y$, every member $F$ of $\Bb(Y) \otimes \Bb(Y)$ may be written in the form
\begin{equation}\label{eqn:aux1}
F = \bigcup_{t \in [0,1]} (A_t \times B_t)
\end{equation}
where $A_t, B_t \in \Bb(Y)$ ($t \in [0,1]$).
\end{lem}
\begin{proof}
Let $\AaA$ be the family of all subsets of $Y \times Y$ which are finite unions of sets of the form $A \times B$ with
$A, B \in \Bb(Y)$. It is easily seen that $\AaA$ is an algebra of subsets of $Y \times Y$. Hence, by the Monotone Class
Theorem (cf. e.g. Theorem~1.3 of \cite{l-l}), $\Bb(Y) \otimes \Bb(Y)$ is the smallest family $\FfF$ such that $\AaA \subset
\FfF$ and
\begin{equation}\label{eqn:aux2}
F_1,F_2,\ldots \in \FfF \implies \bigcap_{n=1}^{\infty} F_n, \bigcup_{n=1}^{\infty} F_n \in \FfF.
\end{equation}
Now let $\FfF$ consists of all sets $F$ which may be written in the form \eqref{eqn:aux1} (with $A_t, B_t \in \Bb(Y)$).
Since $\varempty \in \FfF$, $\AaA \subset \FfF$. So, it remains to check that $\FfF$ satisfies \eqref{eqn:aux2}. Let
$F_j = \bigcup_{t \in [0,1]} A^{(j)}_t \times B^{(j)}_t$ for $j=1,2,\ldots$ It is clear that $\bigcup_{j=1}^{\infty} F_j
\in \FfF$. To this end, put $\Lambda = [0,1]^{\NNN}$ (here $0 \notin \NNN$) and observe that
$$
\bigcap_{j=1}^{\infty} F_j = \bigcup_{\xi \in \Lambda} \bigl(\bigcap_{j=1}^{\infty} A^{(j)}_{\xi(j)}\bigr) \times
\bigl(\bigcap_{j=1}^{\infty} B^{(j)}_{\xi(j)}\bigr)
$$
which finishes the proof since there is a bijection between $\Lambda$ and $[0,1]$.
\end{proof}

As an immediate consequence of the above results, we obtain

\begin{cor}{nomore}
If a topological space $Y$ has \textup{(AEEP)}, then $$\card(Y) \leqsl 2^{\aleph_0}.$$
\end{cor}

Now we want to prove the converse of \COR{nomore} for metrizable $Y$. For need of this, we recall some classical notion
in topology. The \textit{topological cone} $C(Y)$ over a topological space $Y$ is the set $(Y \times (0,1]) \cup
\{\omega_Y\}$ equipped with the topology such that $Y \times (0,1]$ is open in $C(Y)$, the topology on $Y \times (0,1]$
inherited from the one of $C(Y)$ coincides with the product topology, and the sets $(Y \times (0,t)) \cup \{\omega_Y\}$
with $t \in (0,1)$ form a basis of open neighbourhoods of $\omega_Y$ in $C(Y)$. The next result is elementary and we omit
its proof.

\begin{lem}{ext}
If $u\dd X \to Y$ is a continuous function between topological spaces $X$ and $Y$, then the function $\widehat{u}\dd C(X)
\to C(Y)$ given by $\widehat{u}((x,t)) = (u(x),t)$ and $\widehat{u}(\omega_X) = \omega_Y$ is continuous as well.
\end{lem}

An important example of a topological cone is the so-called \textit{hedgehog} space (\cite[Example~4.1.5]{eng}).
The hedgehog $J(\mM)$ of spininess $\mM \geqsl \aleph_0$ is the topological cone over a discrete space of cardinality
$\mM$. Its importance is justified by the following result of Kowalsky \cite{kow} (see also \cite[Theorem~4.4.9]{eng};
it is already known that $[J(\mM)]^{\aleph_0}$ with infinite $\mM$ is homeomorphic to the Hilbert space of Hilbert space
dimension $\mM$, see \cite{to1,to2} or Remark in Exercise~4.4.K of \cite{eng}).

\begin{thm}{hedge}
Every metrizable space of topological weight no greater than $\mM$ (where $\mM \geqsl \aleph_0$) is homeomorphic
to a subset of $[J(\mM)]^{\aleph_0}$.
\end{thm}

As a colorrary of \THM{hedge}, we obtain the following result, which may be interesting in itself.

\begin{pro}{1-1}
Every metrizable space $Y$ such that $\card(Y) \leqsl 2^{\aleph_0}$ admits a continuous one-to-one function of $Y$ into
a separable metrizable space.
\end{pro}
\begin{proof}
Let $D = [0,1]$ and $T = [0,1]$ be equipped with, respectively, the discrete and the natural topology. By \LEM{ext},
the function $C(D) \ni z \mapsto z \in C(T)$ is continuous. This means that there is a one-to-one continuous function
$u\dd J(2^{\aleph_0}) \to W$ where $W$ is a separable metrizable space. But then the function $[J(2^{\aleph_0})]^{\aleph_0}
\ni (x_n)_{n=1}^{\infty} \mapsto (u(x_n))_{n=1}^{\infty} \in W^{\aleph_0}$ is continuous and one-to-one as well. Now it
suffices to apply \THM{hedge}.
\end{proof}

We are now able to prove the main result of the paper.

\begin{thm}{aeep}
For a metrizable space $Y$ \tfcae
\begin{enumerate}[\upshape(i)]
\item $Y$ has \textup{(AEEP)},
\item the topological weight of $Y$ is no greater than $2^{\aleph_0}$,
\item $\card(Y) \leqsl 2^{\aleph_0}$.
\end{enumerate}
\end{thm}
\begin{proof}
The equivalence of (ii) and (iii) follows since $(2^{\aleph_0})^{\aleph_0} = 2^{\aleph_0}$. So, thanks to \LEM{delta} and
\COR{nomore}, we only have to show that $\Delta_Y \in \Bb(Y) \otimes \Bb(Y)$ provided $\card(Y) \leqsl 2^{\aleph_0}$.\par
Assume (iii) is fulfilled. We infer from \PRO{1-1} that there is a separable metrizable space $X$ and a continuous
one-to-one function $u\dd Y \to X$. Then $v = u \times u\dd (Y \times Y,\Bb(Y) \otimes \Bb(Y)) \to (X \times X,\Bb(X)
\otimes \Bb(X))$ is measurable ($v(x,y) = (u(x),u(y))$ for $x,y \in Y$). Since $X$ is separable, $\Bb(X) \otimes \Bb(X) =
\Bb(X \times X)$ and therefore $\Delta_Y = v^{-1}(\Delta_X) \in \Bb(Y) \otimes \Bb(Y)$ and we are done.
\end{proof}

\end{document}